\documentclass[10pt]{article}

\usepackage[margin=1in]{geometry}

\usepackage[T1]{fontenc}
\usepackage[utf8]{inputenc}
\usepackage{lmodern}

\usepackage{amsmath,amssymb,amsthm,mathtools}

\usepackage{enumitem}

\usepackage[hidelinks]{hyperref}

\hypersetup{
	pdfauthor={XuanYu Cao; Chunxiang Wang},
	pdftitle={
		Maximum and Minimum Spectral Radii in an Exceptional Family
		for Edge-Disjoint Spanning Trees
	},
	pdfkeywords={
		spectral radius,
		edge-disjoint spanning trees,
		spanning-tree packing,
		exceptional graph family,
		equitable quotient
	}
}

\usepackage[nameinlink,capitalize,noabbrev]{cleveref}

\theoremstyle{plain}

\newtheorem{theorem}{Theorem}[section]
\newtheorem{lemma}[theorem]{Lemma}
\newtheorem{proposition}[theorem]{Proposition}
\newtheorem{corollary}[theorem]{Corollary}
\newtheorem{claim}[theorem]{Claim}

\theoremstyle{definition}

\theoremstyle{remark}

\newtheorem{remark}[theorem]{Remark}

\crefname{theorem}{Theorem}{Theorems}
\Crefname{theorem}{Theorem}{Theorems}

\crefname{lemma}{Lemma}{Lemmas}
\Crefname{lemma}{Lemma}{Lemmas}

\crefname{proposition}{Proposition}{Propositions}
\Crefname{proposition}{Proposition}{Propositions}

\crefname{corollary}{Corollary}{Corollaries}
\Crefname{corollary}{Corollary}{Corollaries}

\crefname{claim}{Claim}{Claims}
\Crefname{claim}{Claim}{Claims}

\crefname{definition}{Definition}{Definitions}
\Crefname{definition}{Definition}{Definitions}

\crefname{remark}{Remark}{Remarks}
\Crefname{remark}{Remark}{Remarks}

\newcommand{\one}{\mathbf{1}}

\title{Maximum and Minimum Spectral Radii in an Exceptional Family\\
	for Edge-Disjoint Spanning Trees}
\author{XuanYu Cao\thanks{School of Mathematics and
		Statistics, Central China Normal University, Wuhan, China.} \and Chunxiang Wang\thanks{School of Mathematics and Statistics, and Hubei Key
Lab–Math. Sci., Central China Normal University, Wuhan, China}}
\date{}

\begin{document}
	\maketitle
	
	\begin{abstract}
	For a connected graph \(G\), let \(\tau(G)\) denote the maximum number
	of pairwise edge-disjoint spanning trees in \(G\), and let \(\rho(G)\)
	denote its adjacency spectral radius. For integers \(n\ge1\), \(k\ge2\),
	and \(k\le\delta\le2k-1\), let \(\mathcal{G}_{n,\delta}\) be the class
	of connected \(n\)-vertex graphs with minimum degree \(\delta\), and let
	\(\mathcal{L}_{\mathcal{H}}^{2}(n,k,\delta)\subseteq
	\mathcal{G}_{n,\delta}\) denote the exceptional family introduced by
	Chang, Li, and Zhang. They showed that, when \(n\) is sufficiently large
	relative to \(\delta\), determining the sharp adjacency-spectral threshold
	for \(\tau(G)\ge k\) reduces to maximizing \(\rho(G)\) over this family.

	Write \(h=\delta-k\). For every fixed admissible pair \((k,\delta)\)
	and all sufficiently large \(n\), we determine the maximum and minimum
	adjacency spectral radii over
	\(\mathcal{L}_{\mathcal{H}}^{2}(n,k,\delta)\). In the associated
	core--placement representation, write \(M\) for the missing-edge graph
	of the bounded core. For \(h\ge2\), the unique maximizer up to isomorphism
	has \(M\cong K_{1,h}\cup(h+1)K_1\), and its exceptional edges have a
	nested placement on the large-clique side. Consequently, the matching
	missing-edge configuration proposed in Conjecture~2 of
	Chang--Li--Zhang is not extremal. For every \(h\ge1\), the unique
	minimizer up to isomorphism has \(M\cong hK_2\cup2K_1\), and its \(2h\)
	exceptional edges have pairwise distinct endpoints on the large-clique
	side. The boundary cases \(h=0\) and \(h=1\), together with all equality
	cases, are also determined.

	The proof combines an exact core--placement parametrization with a
	uniform Schur-complement resolvent expansion. The first
	candidate-dependent coefficient is an affine function of
	\(\sum_{x\in V(M)}d_M(x)^2\), whereas the first placement-sensitive
	coefficient is a squared-load functional. Explicit equitable quotient
	matrices determine the two extremal radii, and the maximizing graph
	yields the sharp global adjacency-spectral threshold.
\end{abstract}
	
	
	
	\medskip
	
	\section{Introduction}
	
	Throughout this paper, all graphs are finite, simple, and undirected. Unless
otherwise stated, the graph \(G\) under consideration is connected. Let
\(V(G)\), \(E(G)\), \(e(G)\), \(\delta(G)\), \(A(G)\), and \(\rho(G)\)
denote the vertex set, edge set, size, minimum degree, adjacency matrix, and
adjacency spectral radius of \(G\), respectively, and put \(n=|V(G)|\).
Let \(\tau(G)\) denote the maximum number of pairwise edge-disjoint spanning
trees in \(G\). For positive integers \(n\) and \(\delta\), let
\(\mathcal{G}_{n,\delta}\) be the class of connected graphs \(G\) satisfying
\(|V(G)|=n\) and \(\delta(G)=\delta\).

For a partition \(\mathcal P\) of \(V(G)\), let \(e_G(\mathcal P)\) denote
the number of edges of \(G\) joining distinct parts of \(\mathcal P\). The
Nash--Williams--Tutte theorem \cite{NashWilliams1961,Tutte1961} states that,
for every positive integer \(k\), \(\tau(G)\ge k\) if and only if
\(e_G(\mathcal P)\ge k(|\mathcal P|-1)\) for every partition
\(\mathcal P\) of \(V(G)\).
	
	Spectral sufficient conditions for spanning-tree packing have been studied in
several settings. Cioab\u{a} and Wong initiated the adjacency-eigenvalue
approach for regular graphs \cite{CioabaWong2012}. Subsequent extensions to
general graphs were obtained in
\cite{GuLaiLiYao2016,LiShi2013,LiuHongLai2014,LiuHongGuLai2014}; see also
Palmer's survey \cite{Palmer2001}. The results most directly related to the
present work are due to Fan--Gu--Lin \cite{FanGuLin2023} and
Chang--Li--Zhang \cite{ChangLiZhang2026}.

Fan--Gu--Lin considered the range \(\delta\ge2k\). Their sharp size theorem
states that, for \(n\ge2\delta+2\),
\[
	e(G)\ge
	\binom{\delta+1}{2}+\binom{n-\delta-1}{2}+k
	\quad\Longrightarrow\quad
	\tau(G)\ge k.
\]
The bound is sharp: joining \(K_{\delta+1}\) and
\(K_{n-\delta-1}\) by any \(k-1\) edges produces an edge-extremal
obstruction. Among these obstructions, the unique spectral-radius maximizer
up to isomorphism is \(B_{n,\delta+1}^{k-1}\), obtained by making all
\(k-1\) cross edges incident with a single vertex of
\(K_{\delta+1}\). More precisely, their Theorem~1.2 states that if \(G\) is
a connected \(n\)-vertex graph with minimum degree \(\delta\), where
\(k\ge2\), \(\delta\ge2k\), and \(n\ge2\delta+3\), then
\[
	\rho(G)\ge\rho\bigl(B_{n,\delta+1}^{k-1}\bigr)
	\quad\Longrightarrow\quad
	\tau(G)\ge k
	\ \text{or}\
	G\cong B_{n,\delta+1}^{k-1}.
\]
Their proof first reduces the problem to an edge-extremal family and then
uses a Perron-vector edge-moving argument to identify the unique
spectral-radius maximizer.

Chang--Li--Zhang considered the complementary range
\(k\le\delta\le2k-1\). Their sharp size theorem identifies an exceptional
family \(\mathcal{L}_{\mathcal{H}}\), while their spectral theorem shows
that, for \(n\gg\delta\), determining the sharp global spectral threshold
reduces to identifying the spectral-radius maximizers in
\(\mathcal{L}_{\mathcal{H}}^{2}(n,k,\delta)\). In their notation,
Conjecture~2 predicts that the missing-edge graph of such a maximizer, that
is, the complement of its core, is isomorphic to
\((\delta-k)K_2\cup2K_1\), and that all core vertices incident with a
missing edge have a common neighbor in the large clique. Thus the remaining
step between their reduction and an explicit sharp spectral threshold is a
finite but nontrivial optimization problem within
\(\mathcal{L}_{\mathcal{H}}^{2}\).
	
	We complete this optimization and also determine the lower spectral
endpoint. Set
\(
h=\delta-k,
\)
so that \(0\le h\le k-1\). For \(h\ge2\), the missing-edge graph of the
unique maximizer is
\(
K_{1,h}\cup(h+1)K_1,
\)
rather than a matching, and its exceptional \(X\)--\(Y\) edges follow
a nested attachment pattern. By contrast, the unique minimizer has
exactly the matching-type missing-edge graph
\(
hK_2\cup2K_1
\)
proposed by Chang--Li--Zhang, while its exceptional \(X\)--\(Y\) edges
follow a dispersed attachment pattern. These two extremal structures
exhibit a concentration--dispersion principle: concentrating both the
missing-edge degrees and the loads on the \(Y\)-side increases the
spectral radius, whereas dispersing them decreases it.

Our proof method differs from the local Perron-vector edge-moving
arguments used in the aforementioned works. We eliminate the
large-clique block via a Schur-complement reduction and then derive a
uniform resolvent expansion whose successive coefficients are
determined by statistics of the bounded core and its attachment to the
clique. The first coefficient depending on the missing-edge graph \(M\)
is
\(
\|b_M\|_2^2=\sum_{v\in V(M)}d_M(v)^2,
\)
the degree-square sum of \(M\). Once \(M\) is fixed, the first
placement-dependent coefficient is a quadratic load functional on the
\(Y\)-side. The resolvent expansion provides uniform asymptotic
comparisons among all admissible candidates and treats the maximization
and minimization problems within a single framework.
	
	More specifically, our contributions are threefold:
\begin{enumerate}[label=\textup{(\roman*)},leftmargin=*]
	\item we give an exact parametrization of
	\(\mathcal{L}_{\mathcal{H}}^{2}(n,k,\delta)\) in terms of the graphs
	\(G_h(M,B)\), together with a complete characterization of the
	degree and edge-count constraints;
	\item we establish a uniform resolvent expansion for a graph
	consisting of a large clique attached to a bounded exceptional core;
	and
	\item we determine the unique maximizer and minimizer, up to
	isomorphism, and hence characterize all equality cases. We then
	combine the resulting maximum with the reduction theorem of
	Chang--Li--Zhang \cite{ChangLiZhang2026} to obtain a sharp global
	adjacency-spectral threshold.
\end{enumerate}

The comparisons used to establish these classifications are asymptotic
as \(n\to\infty\). For each fixed admissible pair \((k,\delta)\), the
finiteness of the core--placement patterns and the uniformity of the
remainder estimates ensure the existence of an integer
\(
n_{\mathrm{int}}(k,\delta)\ge2\delta+2
\)
such that both extremal classifications within
\(\mathcal{L}_{\mathcal{H}}^{2}(n,k,\delta)\) hold for every
\(n\ge n_{\mathrm{int}}(k,\delta)\); see
\Cref{subsec:common-threshold}. We do not claim an explicit formula
for this onset threshold. Let \(n_{\mathrm{CLZ}}(k,\delta)\) be an
integer such that Theorem~1.5 of Chang--Li--Zhang applies for every
\(n\ge n_{\mathrm{CLZ}}(k,\delta)\). For the final global statement,
we define
\begin{equation}
	\label{eq:n0-definition}
	n_0(k,\delta)=
	\max\{n_{\mathrm{CLZ}}(k,\delta),
	      n_{\mathrm{int}}(k,\delta)\}.
\end{equation}
Since Chang--Li--Zhang state their hypothesis in the nonquantitative
form \(n\gg\delta\), the present argument does not yield an explicit
numerical value of \(n_0(k,\delta)\). This lack of effectivity concerns
only the onset threshold, not the spectral value of the candidate
itself. Indeed, for every \(n\) for which the candidate is defined, it
admits an explicit equitable quotient matrix whose Perron root is
exactly the adjacency spectral radius of the candidate.
	
	We now state the two extremal comparison results within the exceptional
family. We begin with the maximization result.

\begin{theorem}
	\label{thm:intro-ordinary}
	Fix \(k\ge2\) and \(k\le\delta\le2k-1\), and set
	\(
	h=\delta-k.
	\)
	For every \(n\ge n_{\mathrm{int}}(k,\delta)\), there exists a graph
	\(
		G^\star_{h,k,n}
		\in\mathcal{L}_{\mathcal{H}}^{2}(n,k,\delta),
	\)
	characterized below, such that, for every
	\(
	G\in\mathcal{L}_{\mathcal{H}}^{2}(n,k,\delta),
	\)
	\(
		\rho(G)\le \rho\bigl(G^\star_{h,k,n}\bigr).
	\)
	Equality holds if and only if
	\(
	G\cong G^\star_{h,k,n}.
	\)
	Consequently, \(G^\star_{h,k,n}\) is the unique maximizer up to
	isomorphism. Its structure is as follows.
	\begin{enumerate}[label=\textup{(\roman*)}]
		\item If \(h=0\), then
		\(\mathcal{L}_{\mathcal{H}}^{2}(n,k,\delta)\)
		consists of a single isomorphism class, represented by
		\(G^\star_{0,k,n}\).

		\item If \(h=1\), then the corresponding missing-edge graph
		\(M^\star\) has a unique edge \(uv\). After relabeling the
		vertices of \(Y\), its \(Y\)-side neighborhoods satisfy
		\(
			N_Y(u)=N_Y(v)=\{y_1\}.
		\)

		\item If \(h\ge2\), then the corresponding missing-edge graph is
		\(
			M^\star\cong K_{1,h}\cup(h+1)K_1.
		\)
		Let \(c\) be the center of the \(K_{1,h}\) component, and let
		\(\ell_1,\ldots,\ell_h\) be its leaves. After relabeling the
		vertices of \(Y\), its \(Y\)-side neighborhoods satisfy
		\(
			N_Y(c)=\{y_1,\ldots,y_h\},
			\;
			N_Y(\ell_i)=\{y_1\}
			\; (1\le i\le h).
		\)
	\end{enumerate}
\end{theorem}
	
  \begin{corollary}[Failure of Conjecture~2 of Chang--Li--Zhang for \(h\ge2\)]
	\label{cor:conjecture-two}
	Let \(k\) and \(h\) be integers satisfying
	\(
	k\ge3
	\)
	and
	\(
	2\le h\le k-1,
	\)
	and set \(\delta=k+h\). For every
	\(
	n\ge n_{\mathrm{int}}(k,\delta),
	\)
	the matching-core candidate proposed in Conjecture~2 of
	Chang--Li--Zhang does not maximize the adjacency spectral radius over
	\(\mathcal{L}_{\mathcal{H}}^{2}(n,k,\delta)\). Instead, the unique
	maximizer is \(G^\star_{h,k,n}\), whose missing-edge graph is
	\(
		K_{1,h}\cup(h+1)K_1.
	\)
	Consequently, Conjecture~2 is false for every parameter pair
	\((k,\delta)\) satisfying
	\(
		k\ge3,
		\;
		k+2\le\delta\le2k-1.
	\)
\end{corollary}
    
  \begin{proof}
	Conjecture~2 predicts that a maximizing graph has missing-edge graph
	\(
		hK_2\cup2K_1,
	\)
	whose maximum degree is \(1\). By
	\Cref{thm:intro-ordinary}, however, the unique maximizer has
	missing-edge graph
	\(
		K_{1,h}\cup(h+1)K_1,
	\)
	whose maximum degree is \(h\). Equivalently,
	\(
		\Delta\bigl(hK_2\cup2K_1\bigr)=1,
		\;
		\Delta\bigl(K_{1,h}\cup(h+1)K_1\bigr)=h.
	\)
	Since \(h\ge2\), these two missing-edge graphs are nonisomorphic.
	Hence the unique maximizer does not have the structure predicted by
	Conjecture~2.
\end{proof}

\begin{remark}
	When \(h=1\), the two missing-edge graphs both specialize to
	\(
		K_2\cup2K_1.
	\)
	Moreover, the common-neighbor attachment pattern on the \(Y\)-side
	proposed by Chang--Li--Zhang, namely
	\(
		N_Y(u)=N_Y(v)=\{y_1\},
	\)
	agrees exactly with \Cref{thm:intro-ordinary}(ii).
\end{remark}
	
	The next theorem identifies the lower spectral endpoint. In particular,
for \(h\ge2\), the missing-edge graph at this endpoint is precisely the
matching-type graph proposed by Chang--Li--Zhang for the upper endpoint.

\begin{theorem}
	\label{thm:intro-ordinary-minimum}
	Fix integers \(k\) and \(\delta\) satisfying
	\(
	k\ge2
	\)
	and
	\(
	k\le\delta\le2k-1,
	\)
	and set \(h=\delta-k\). For every integer
	\(n\ge n_{\mathrm{int}}(k,\delta)\), there exists a graph
	\(
		G^{\min}_{h,k,n}
		\in\mathcal{L}_{\mathcal{H}}^{2}(n,k,\delta)
	\)
	such that, for every
	\(
	G\in\mathcal{L}_{\mathcal{H}}^{2}(n,k,\delta),
	\)
	we have
	\(
		\rho(G)\ge \rho\bigl(G^{\min}_{h,k,n}\bigr).
	\)
	Equality holds if and only if
	\(
	G\cong G^{\min}_{h,k,n}.
	\)
	Consequently, \(G^{\min}_{h,k,n}\) is the unique minimizer up to
	isomorphism.

	If \(h=0\), then
	\(
		G^{\min}_{0,k,n}\cong G^\star_{0,k,n}.
	\)
	If \(h\ge1\), then the missing-edge graph of
	\(G^{\min}_{h,k,n}\) is isomorphic to
	\(
		hK_2\cup2K_1,
	\)
	and the \(Y\)-endpoints of its \(2h\) exceptional \(X\)--\(Y\)
	edges are pairwise distinct.
\end{theorem}

The remainder of the paper is organized as follows. Section~2 establishes
a precise correspondence between the framework of Chang--Li--Zhang and
our core--placement model, develops the required spectral tools, proves
the two extremal theorems, and derives the sharp global
adjacency-spectral threshold. Section~3 concludes the paper.

	\section{Spectral Optimization in the Exceptional Family}
	\label{sec:ordinary}
	
	\subsection{Reduction and parametrization}
\label{subsec:clz-reduction}

Recall that, for a connected graph \(G\), \(\tau(G)\) denotes the maximum
number of pairwise edge-disjoint spanning trees in \(G\). For a partition
\(\mathcal P\) of \(V(G)\), let \(e_G(\mathcal P)\) denote the number of
edges of \(G\) whose endpoints lie in distinct parts of \(\mathcal P\).
The Nash--Williams--Tutte theorem
\cite{NashWilliams1961,Tutte1961} states that
\(\tau(G)\ge k\) if and only if every partition \(\mathcal P\) of \(V(G)\)
satisfies
\(
	e_G(\mathcal P)\ge k\bigl(|\mathcal P|-1\bigr).
\)

We follow the notation of Chang--Li--Zhang
\cite{ChangLiZhang2026}. Fix integers \(k\) and \(\delta\) satisfying
\(
	k\ge2,
	\;
	k\le\delta\le2k-1,
\)
and set
\(
h=\delta-k.
\)
Thus \(0\le h\le k-1\).

Let \(\mathcal{H}_{k,\delta}\) be the family of connected graphs on
\(2h+2\) vertices with at least
\(
	\binom{2h+2}{2}-h
\)
edges. Our optimization concerns the subfamily for which the core has
the minimum permitted number of edges.

More precisely, the graphs in
\(\mathcal{L}_{\mathcal{H}}^{2}(n,k,\delta)\) are constructed as follows.
Take pairwise disjoint vertex sets \(X,W,Y\) satisfying
\(
	|X|=2h+2,
	\;
	|W|=k-h-1,
	\;
	|Y|=n-k-h-1.
\)
Choose \(H\in\mathcal{H}_{k,\delta}\) with \(V(H)=X\) and
\(
	e(H)=\binom{2h+2}{2}-h.
\)
Starting from
\(
	K_W\vee\bigl(H\cup K_Y\bigr),
\)
add an edge set
\(
B\subseteq E(K_{X,Y})
\)
with \(|B|=2h\), subject to the requirement that every vertex of \(X\)
has degree at least \(\delta\) in the resulting graph \(G\); equivalently,
\(
	d_G(x)\ge\delta
	\; \text{for every }x\in X.
\)
Here \(K_W\) and \(K_Y\) denote the complete graphs on \(W\) and \(Y\),
respectively, \(K_{X,Y}\) denotes the complete bipartite graph with parts
\(X\) and \(Y\), and \(\vee\) denotes the graph join. We interpret
\(K_\varnothing\) as the empty graph. This construction gives precisely
the subfamily denoted by
\(\mathcal{L}_{\mathcal{H}}^{2}(n,k,\delta)\) in
\cite{ChangLiZhang2026}.

Define the family of spectral-radius maximizers by
\[
\mathcal{L}^{*}(n,k,\delta)
=
\left\{
	G\in\mathcal{L}_{\mathcal{H}}^{2}(n,k,\delta):
	\rho(G)
	=
	\max_{J\in\mathcal{L}_{\mathcal{H}}^{2}(n,k,\delta)}
	\rho(J)
\right\}.
\]

Theorem~1.4 of Chang--Li--Zhang states that, for
\(n\ge2\delta+2\), every \(G\in\mathcal{G}_{n,\delta}\) satisfying
\(\tau(G)\le k-1\) obeys
\[
e(G)\le
\binom{n+2k-2\delta-2}{2}
+\binom{2(\delta-k+1)}{2}
+(2k-\delta-1)(2\delta-2k+2)
+\delta-k,
\]
with equality if and only if
\(
	G\in\mathcal{L}_{\mathcal{H}}(n,k,\delta).
\)

By the definition of \(n_{\mathrm{CLZ}}(k,\delta)\),
Theorem~1.5 of Chang--Li--Zhang gives the following spectral reduction.
For every \(n\ge n_{\mathrm{CLZ}}(k,\delta)\), if
\(
	G\in\mathcal{G}_{n,\delta}
	\; \text{and}\;
	\widehat G\in\mathcal{L}^{*}(n,k,\delta),
\)
then
\(
	\rho(G)\ge\rho(\widehat G)
	\;\Longrightarrow\;
	\tau(G)\ge k
	\ \text{or}\
	G\in\mathcal{L}^{*}(n,k,\delta).
\)
Consequently, it remains to determine
\(\mathcal{L}^{*}(n,k,\delta)\) explicitly and to compute its common
maximum spectral radius.

    \subsubsection{Core--placement representation}
    \label{subsec:exceptional-family-dictionary}
    
    The notation in \cite{ChangLiZhang2026} is convenient for the global
    reduction, whereas the internal optimization becomes more transparent
    once the complement of the core is made explicit. With
    \(h=\delta-k\), the correspondence is
    \[
    \begin{array}{c|c|c}
    	\text{Chang--Li--Zhang object}
    	&
    	\text{notation in this paper}
    	&
    	\text{order or role}
    	\\ \hline
    	V\bigl(H_{2(\delta-k+1)}\bigr)
    	&
    	X
    	&
    	|X|=2h+2
    	\\
    	V\bigl(K_{2k-\delta-1}\bigr)
    	&
    	W
    	&
    	|W|=k-h-1
    	\\
    	V\bigl(K_{n-\delta-1}\bigr)
    	&
    	Y
    	&
    	|Y|=n-k-h-1
    	\\
    	H_{2(\delta-k+1)}^{c}
    	&
    	M
    	&
    	\text{missing-edge graph on \(X\)}
    	\\
    	E_G(X,Y)
    	&
    	B
    	&
    	\text{bipartite graph of exceptional \(X\)--\(Y\) edges}.
    \end{array}
    \]
    Here the complement in the fourth row is taken with respect to the
    complete graph on \(X\). The next proposition makes this correspondence
    precise and records the numerical constraints used throughout the paper.
    
   \begin{proposition}
	\label{prop:LH2-reparametrization}
	Fix integers \(k\) and \(\delta\) satisfying
	\(
		k\ge2,
		\;
		k\le\delta\le2k-1,
	\)
	set \(h=\delta-k\), and let \(n\ge2\delta+2\). Every graph
	\(
	G\in\mathcal{L}_{\mathcal{H}}^{2}(n,k,\delta)
	\)
	admits a partition
	\(
		V(G)=X\mathbin{\dot\cup}W\mathbin{\dot\cup}Y
	\)
	satisfying
	\(
		|X|=2h+2,
		\;
		|W|=k-h-1,
		\;
		|Y|=n-k-h-1,
	\)
	such that \(W\cup Y\) induces a clique, \(W\) is complete to \(X\),
	and
	\(
		E(G[X])=E(K_X)\setminus E(M)
	\)
	for a simple graph \(M\) on \(X\).

	Let \(B\) be the bipartite graph on \(X\cup Y\) with bipartition
	\((X,Y)\) and edge set \(E_G(X,Y)\). Then
	\(
		e(M)=h,
		\;
		e(B)=2h,
		\;
		d_B(x)=d_M(x)
		\; \text{for every }x\in X.
	\)

	Conversely, suppose that pairwise disjoint sets \(X,W,Y\) have the
	displayed cardinalities, and that \(M\) and \(B\) satisfy the
	conditions above. Define \(G_h(M,B)\) on
	\(X\mathbin{\dot\cup}W\mathbin{\dot\cup}Y\) by requiring that
	\(
		E\bigl(G_h(M,B)[X]\bigr)
		=E(K_X)\setminus E(M),
	\)
	that \(W\cup Y\) induce a clique, that \(W\) be complete to \(X\),
	and that the \(X\)--\(Y\) edges be precisely the edges of \(B\).
	Then
	\(
		G_h(M,B)
		\in\mathcal{L}_{\mathcal{H}}^{2}(n,k,\delta).
	\)
	Moreover, every vertex of \(X\) has degree exactly \(\delta\), whereas
	every vertex of \(W\cup Y\) has degree at least \(\delta\).
	Consequently,
	\(
		\delta\bigl(G_h(M,B)\bigr)=\delta.
	\)
\end{proposition}

\begin{proof}
	By the defining construction of
	\(\mathcal{L}_{\mathcal{H}}^{2}(n,k,\delta)\),
	\(
		|X|
		=
		2(\delta-k+1)
		=
		2h+2,
	\)
	and
	\(
		|W|
		=
		2k-\delta-1
		=
		k-h-1,
		\;
		|Y|
		=
		n-\delta-1
		=
		n-k-h-1.
	\)
	The join in the Chang--Li--Zhang construction implies that
	\(W\cup Y\) induces a clique and that \(W\) is complete to \(X\).

	Define the missing-edge graph \(M\) on \(X\) by
	\(
		E(M)=E(K_X)\setminus E(G[X]).
	\)
	Since the core has the minimum admissible number of edges,
	\(
		e(G[X])=\binom{2h+2}{2}-h,
	\)
	and hence \(e(M)=h\). For every \(x\in X\),
	\begin{align*}
		d_G(x)
		&=
		\bigl(2h+1-d_M(x)\bigr)
		+(k-h-1)
		+d_B(x)\\
		&=
		\delta+d_B(x)-d_M(x).
	\end{align*}
	Since \(d_G(x)\ge\delta\), we have
	\(
		d_B(x)\ge d_M(x)
		\; \text{for every }x\in X.
	\)
	Each edge of \(B\) has exactly one endpoint in \(X\); therefore,
	\(
		\sum_{x\in X}d_B(x)=e(B).
	\)
	Together with the handshaking lemma applied to \(M\), this gives
	\[
		\sum_{x\in X}\bigl(d_B(x)-d_M(x)\bigr)
		=
		e(B)-2e(M)
		=
		2h-2h
		=
		0.
	\]
	Each summand is nonnegative, so every summand must be zero.
	Consequently,
	\(
		d_B(x)=d_M(x)
		\; \text{for every }x\in X.
	\)

	Conversely, suppose that \(X,W,Y,M,B\) satisfy all the stated
	conditions, and consider \(G_h(M,B)\). Set
	\(
		H=K_X-M.
	\)
	The edge-connectivity of \(K_{2h+2}\) is \(2h+1\). Since
	\(h<2h+1\), deleting the \(h\) edges of \(M\) cannot disconnect
	\(K_{2h+2}\). Thus \(H\) is connected. Moreover,
	\(
		e(H)=\binom{2h+2}{2}-h,
	\)
	so \(H\in\mathcal{H}_{k,\delta}\).

	The graph \(G_h(M,B)\) is connected. Indeed, if \(h\ge1\), then
	\(e(B)=2h>0\), so at least one edge of \(B\) connects the connected
	graph \(H\) to the clique induced by \(W\cup Y\). If \(h=0\), then
	\(
		|W|=k-1>0,
	\)
	and every vertex of \(W\) is adjacent to every vertex of \(X\cup Y\).

	For every \(x\in X\), the equality \(d_B(x)=d_M(x)\) gives
	\begin{align*}
		d_{G_h(M,B)}(x)
		&=
		\bigl(2h+1-d_M(x)\bigr)
		+(k-h-1)
		+d_B(x)\\
		&=
		k+h
		=
		\delta.
	\end{align*}
	Every vertex of \(W\) is universal. For every \(y\in Y\),
	\(
		d_{G_h(M,B)}(y)
		\ge
		|W|+|Y|-1
		=
		n-2h-3.
	\)
	Since \(n\ge2\delta+2=2k+2h+2\) and \(h\le k-1\),
	\(
		n-2h-3
		\ge
		2k-1
		\ge
		k+h
		=
		\delta.
	\)
	Thus every vertex outside \(X\) has degree at least \(\delta\), while
	every vertex of \(X\) has degree exactly \(\delta\). It follows that
	\(
		\delta\bigl(G_h(M,B)\bigr)=\delta.
	\)
	All the defining conditions of
	\(\mathcal{L}_{\mathcal{H}}^{2}(n,k,\delta)\) are therefore
	satisfied.
\end{proof}
    
   For fixed labeled sets \(X,W,Y\), the preceding construction gives a
bijection between the admissible pairs \((M,B)\) and the graphs obtained
on this prescribed vertex partition. Indeed, \(M\) and \(B\) are uniquely
recovered from \(G\) by
\(
	E(M)=E(K_X)\setminus E(G[X]),
	\;
	V(B)=X\cup Y,
	\;
	E(B)=E_G(X,Y).
\)
We therefore denote the graph corresponding to an admissible pair
\((M,B)\) by
\(
	G_h(M,B).
\)
Relabelings within \(X\) and \(Y\), applied consistently to \(M\) and
\(B\), yield isomorphic graphs \(G_h(M,B)\); the vertices of \(W\) may
also be relabeled arbitrarily.

\begin{remark}[Finiteness of the core--placement patterns]
	\label{rem:finite-patterns}
	For fixed \(h\) and \(k\), only finitely many patterns \((M,B)\)
	must be compared up to relabeling, and their number is independent
	of \(n\) for \(n\ge2\delta+2\). Indeed, \(M\) has \(h\) edges on the
	fixed vertex set \(X\), where
	\(
		|X|=2h+2.
	\)
	Define the active and inactive subsets of \(Y\) by
	\(
		Y_{\mathrm{act}}
		=
		\{y\in Y:d_B(y)>0\},
		\;
		Y_0=Y\setminus Y_{\mathrm{act}}.
	\)
	Since \(e(B)=2h\), we have
	\(
		|Y_{\mathrm{act}}|\le2h.
	\)
	Thus the nonisolated part of \(B\) is a bipartite graph whose two
	parts have orders \(2h+2\) and at most \(2h\), respectively.

	Every vertex of \(Y_0\) has no neighbors in \(X\). Since \(W\cup Y\)
	induces a clique, every \(y\in Y_0\) satisfies
	\(
		N_{G_h(M,B)}[y]=W\cup Y.
	\)
	Consequently, \(Y_0\) is a true-twin class in \(G_h(M,B)\).

	The standing assumption \(n\ge2\delta+2\) gives
	\(
		|Y|
		=
		n-k-h-1
		\ge
		k+h+1
		\ge
		2h+2,
	\)
	where the last inequality follows from \(h\le k-1\). Hence \(Y\)
	always contains enough vertices to realize every admissible
	attachment pattern.

	Up to relabeling, an admissible pair is therefore determined by a
	graph \(M\) on \(2h+2\) vertices and the nonisolated part of a
	bipartite graph \(B\) whose \(Y\)-side has order at most \(2h\).
	There are only finitely many such pairs, with their number bounded
	by a constant depending only on \(h\). Increasing \(n\) merely adds
	vertices to the inactive true-twin class \(Y_0\). This finiteness
	underlies the uniform comparisons developed below.
\end{remark}
	
	\begin{proposition}
	\label{prop:packing-obstruction}
	Fix integers \(k\) and \(\delta\) satisfying
	\(
		k\ge2,
		\;
		k\le\delta\le2k-1,
	\)
	set \(h=\delta-k\), and let \(n\ge2\delta+2\). Then every graph
	\(
		G\in\mathcal{L}_{\mathcal{H}}^{2}(n,k,\delta)
	\)
	satisfies
	\(
		\tau(G)\le k-1.
	\)
	More precisely, there exists a partition \(\mathcal P_X\) of \(V(G)\)
	such that
	\(
		e_G(\mathcal P_X)
		=
		k\bigl(|\mathcal P_X|-1\bigr)-1.
	\)
\end{proposition}

\begin{proof}
	By \Cref{prop:LH2-reparametrization}, choose a representation
	\(
		G=G_h(M,B)
	\)
	with associated partition
	\(
		V(G)=X\mathbin{\dot\cup}W\mathbin{\dot\cup}Y.
	\)
	Set
	\(
		t:=|X|=2h+2,
		\;
		a:=|W|=k-h-1,
		\;
		C:=W\cup Y.
	\)
	Notice that
	\(
		|Y|=n-\delta-1\ge\delta+1>0,
	\)
	so \(C\neq\varnothing\). Consider the partition
	\(
		\mathcal P_X
		=
		\{C\}\cup\bigl\{\{x\}:x\in X\bigr\}.
	\)
	Every edge of \(G[X]\), as well as every edge between \(X\) and
	\(C\), crosses \(\mathcal P_X\), whereas no edge with both endpoints
	in \(C\) crosses the partition. Consequently,
	\begin{align*}
		e_G(\mathcal P_X)
		&=
		e(G[X])+e_G(X,C)\\
		&=
		\left(\binom{t}{2}-h\right)
		+\bigl(at+e(B)\bigr)\\
		&=
		\left(\binom{t}{2}-h\right)
		+(at+2h)\\
		&=
		\binom{2h+2}{2}
		+(k-h-1)(2h+2)+h\\
		&=
		k(2h+2)-1.
	\end{align*}
	Since
	\(
		|\mathcal P_X|-1=|X|=2h+2,
	\)
	it follows that
	\(
		e_G(\mathcal P_X)
		=
		k\bigl(|\mathcal P_X|-1\bigr)-1
		<
		k\bigl(|\mathcal P_X|-1\bigr).
	\)
	Thus \(\mathcal P_X\) violates the Nash--Williams--Tutte condition
	for the existence of \(k\) pairwise edge-disjoint spanning trees.
	Therefore \(\tau(G)<k\). Since \(\tau(G)\) is an integer, we conclude
	that
	\(
		\tau(G)\le k-1.
	\)
\end{proof}
	
   \subsubsection{Quotient matrices of equitable partitions}
\label{subsec:spectral-tools}

We record the following standard lemma on equitable partitions, which
will be used later to compute the two extremal spectral radii; see also
\cite{CvetkovicRowlinsonSimic2010,ChangLiZhang2026}.

\begin{lemma}
	\label{lem:equitable-quotient}
	Let \(G\) be a connected graph, and let
	\(
		V(G)=V_1\mathbin{\dot\cup}\cdots\mathbin{\dot\cup}V_m
	\)
	be an equitable partition into nonempty classes. Suppose that every
	vertex in \(V_i\) has exactly \(q_{ij}\) neighbors in \(V_j\), and let
	\(
		Q=(q_{ij})_{1\le i,j\le m}
	\)
	be the corresponding quotient matrix. Then
	\(
		\rho(G)=\rho(Q).
	\)
\end{lemma}

\begin{proof}
	Since \(G\) is connected, the quotient graph obtained by contracting
	each class \(V_i\) to a single vertex is connected. Hence the support
	digraph of the nonnegative matrix \(Q\) is strongly connected, and
	therefore \(Q\) is irreducible.

	Let
	\(
		\mathbf z=(z_1,\ldots,z_m)^{\mathsf T}>0
	\)
	be a Perron vector of \(Q\), so that
	\(
		Q\mathbf z=\rho(Q)\mathbf z.
	\)
	Define \(\widetilde{\mathbf z}\in\mathbb R^{V(G)}\) by
	\(
		\widetilde z_v=z_i
		\; \text{whenever }v\in V_i.
	\)
	For every \(v\in V_i\), equitability gives
	\[
		\bigl(A(G)\widetilde{\mathbf z}\bigr)_v
		=
		\sum_{j=1}^{m}q_{ij}z_j
		=
		(Q\mathbf z)_i
		=
		\rho(Q)z_i
		=
		\rho(Q)\widetilde z_v.
	\]
	Consequently,
	\(
		A(G)\widetilde{\mathbf z}
		=
		\rho(Q)\widetilde{\mathbf z}.
	\)
	Since \(\widetilde{\mathbf z}>0\) and \(A(G)\) is irreducible, the
	Perron--Frobenius theorem implies that the corresponding eigenvalue
	\(\rho(Q)\) is the spectral radius of \(A(G)\). Therefore
	\(
		\rho(G)=\rho(Q).
	\)
\end{proof}
    
	\subsection{Spectral expansion for a bounded core attached to a large clique}
\label{subsec:large-clique-expansion}

For every finite index set \(U\), let \(\one_U\) denote the all-ones
column vector indexed by \(U\). For each positive integer \(s\), let
\(I_s\) and \(J_s\) denote the \(s\times s\) identity matrix and the
\(s\times s\) all-ones matrix, respectively.

Let \(G_h(M,B)\) be the graph associated with an admissible pair
\((M,B)\) as in \Cref{prop:LH2-reparametrization}. Set
\(
	t:=|X|=2h+2,
	\;
	a:=|W|=k-h-1,
	\;
	C:=W\cup Y,
	\;
	N:=|C|=n-2h-2.
\)
For \(x\in X\), define
\(
	d_C(x):=|N_G(x)\cap C|.
\)
By \Cref{prop:LH2-reparametrization},
\(
	d_C(x)
	=
	|W|+d_B(x)
	=
	a+d_M(x).
\)
We therefore define \(b_M\in\mathbb R^X\) by
\begin{equation}
	\label{eq:bM}
	b_M(x)
	=
	d_C(x)
	=
	a+d_M(x)
	=
	k-h-1+d_M(x).
\end{equation}

Let
\(
	A_M:=A(K_X-M),
\)
and let \(R\) be the \(N\times t\) bipartite adjacency matrix between
\(C\) and \(X\), with rows indexed by \(C\) and columns indexed by \(X\).
Thus
\[
	R_{cx}
	=
	\begin{cases}
		1, & cx\in E(G_h(M,B)),\\
		0, & cx\notin E(G_h(M,B)).
	\end{cases}
\]
It follows immediately that
\(
	R^{\mathsf T}\one_C=b_M.
\)

Unless explicitly indicated by a subscript \(F\), all matrix norms in
this subsection are operator \(2\)-norms, and all vector norms are
Euclidean norms. Since \(A_M\) is the adjacency matrix of a graph on
\(t\) vertices,
\(
	\|A_M\|
	\le
	\Delta(K_X-M)
	\le
	t-1.
\)
Moreover, \(R\) has exactly
\(
	at+e(B)=at+2h
\)
nonzero entries. Hence
\(
	\|R\|
	\le
	\|R\|_F
	=
	\sqrt{at+2h}.
\)
Finally, since \(M\) has \(h\) edges, \(d_M(x)\le h\) for every
\(x\in X\). Therefore
\(
	\|R^{\mathsf T}\one_C\|
	=
	\|b_M\|
	\le
	\sqrt{t}\,(a+h).
\)
Consequently,
\(
	\|A_M\|
	+\|R\|
	+\|R^{\mathsf T}\one_C\|
	\le
	K_{h,k},
\)
where
\begin{equation}
	\label{eq:explicit-uniform-K}
	K_{h,k}
	:=
	(t-1)+\sqrt{at+2h}+\sqrt{t}\,(a+h).
\end{equation}
The constant \(K_{h,k}\) is independent of \(M\), \(B\), and \(n\).
Thus the hypotheses of the next lemma hold uniformly throughout the
exceptional family.

Here and below, the notation
\(
	E_N=O_{t,K}(N^{-r})
\)
uniformly over the exceptional family means that there exist constants
\(C_r(t,K)>0\) and \(N_r(t,K)\) such that
\(
	|E_N|
	\le
	C_r(t,K)N^{-r}
\)
for every admissible pair \((M,B)\) and every
\(N\ge N_r(t,K)\). For vector- or matrix-valued remainders, the absolute
value is replaced by the corresponding Euclidean or operator norm.
The constants \(C_r(t,K)\) and \(N_r(t,K)\) are independent of
\(M\), \(B\), and \(n\).
	
	\begin{lemma}[A large clique with a bounded exceptional part]
	\label{lem:large-clique}
	Let \(G_N\) be a connected graph with a partition
	\(
		V(G_N)=C\mathbin{\dot\cup}X,
	\)
	where \(C\) induces a clique of order \(N\) and
	\(
	|X|=t\ge1
	\)
	is fixed. Ordering the vertices of \(C\) first and those of \(X\)
	second, suppose that
	\[
		A(G_N)
		=
		\begin{pmatrix}
			J_N-I_N & R\\
			R^{\mathsf T} & A_X
		\end{pmatrix}.
	\]
	Assume that
	\(
		\|R\|+\|A_X\|+\|R^{\mathsf T}\one_C\|\le K,
	\)
	where \(K\) is independent of \(N\), and set
	\(
		b:=R^{\mathsf T}\one_C.
	\)
	Then, uniformly over all such \(R\) and \(A_X\), whenever
	\(
		N\ge5\max\{K,1\}+2,
	\)
	we have
	\begin{align}
		\rho(G_N)
		={}&
		N-1
		+\frac{b^{\mathsf T}b}{N(N-1)}
		+\frac{b^{\mathsf T}A_Xb}{N(N-1)^2}
		+O_{t,K}(N^{-4}),
		\label{eq:large-clique-expansion}\\
		\rho(G_N)
		={}&
		N-1
		+\frac{b^{\mathsf T}b}{N(N-1)}
		+\frac{b^{\mathsf T}A_Xb}{N(N-1)^2}
		+\frac{b^{\mathsf T}A_X^2b}{N(N-1)^3}
		\notag\\
		&\quad
		+\frac{
			b^{\mathsf T}R^{\mathsf T}\mathcal P_C Rb
		}{
			N^2(N-1)^2
		}
		+O_{t,K}(N^{-5}),
		\label{eq:fourth-order}
	\end{align}
	where
	\(
		\mathcal P_C
		:=
		I_N-\frac{1}{N}\one_C\one_C^{\mathsf T}
	\)
	is the orthogonal projection onto \(\one_C^\perp\).

	In particular, the first correction term that may depend on the
	candidate is
	\(
		\frac{b^{\mathsf T}b}{N(N-1)}.
	\)
	Once \(A_X\) and \(b\) are fixed, the first quantity that may depend
	on the placement encoded by \(R\) is
	\(
		b^{\mathsf T}R^{\mathsf T}\mathcal P_C Rb
		=
		\|\mathcal P_C Rb\|^2.
	\)
\end{lemma}
	
	\begin{proof}
	Write
	\[
	A_0
	=
	\begin{pmatrix}
		J_N-I_N & 0\\
		0 & 0
	\end{pmatrix},
	\qquad
	E
	=
	\begin{pmatrix}
		0 & R\\
		R^{\mathsf T} & A_X
	\end{pmatrix}.
	\]
	The spectrum of \(A_0\) consists of the eigenvalue \(N-1\), the
	eigenvalue \(-1\) with multiplicity \(N-1\), and the eigenvalue \(0\)
	with multiplicity \(t\). In particular, since \(t\ge1\),
	\(
	\lambda_2(A_0)=0.
	\)
	Moreover,
	\begin{align*}
		\|E\|
		&\le
		\left\|
		\begin{pmatrix}
			0 & R\\
			R^{\mathsf T} & 0
		\end{pmatrix}
		\right\|
		+
		\left\|
		\begin{pmatrix}
			0 & 0\\
			0 & A_X
		\end{pmatrix}
		\right\|\\
		&=
		\|R\|+\|A_X\|\\
		&\le K.
	\end{align*}
	By Weyl's inequality,
	\(
		\lambda_1(A_0+E)
		\in[N-1-K,N-1+K],
	\)
	whereas
	\(
		\lambda_2(A_0+E)
		\le
		\lambda_2(A_0)+\|E\|
		\le K.
	\)
	If \(N>2K+1\), then
	\(
		N-1-K>K,
	\)
	so the largest eigenvalue is separated from the remainder of the
	spectrum. Denote this eigenvalue by
	\(
		\lambda=\rho(G_N).
	\)
	In particular,
	\(
		\lambda=N-1+O_{t,K}(1).
	\)

	Let \((u,y)^{\mathsf T}\) be a positive Perron vector of \(G_N\),
	normalized so that
	\(
		\frac{1}{N}\one_C^{\mathsf T}u=1.
	\)
	Write
	\(
		u=\one_C+z,
		\;
		\one_C^{\mathsf T}z=0.
	\)
	The two block equations are
	\(
		(J_N-I_N)u+Ry=\lambda u
	\)
	and
	\(
		R^{\mathsf T}u+A_Xy=\lambda y.
	\)
	Since
	\(
		(J_N-I_N)(\one_C+z)
		=
		(N-1)\one_C-z,
	\)
	projecting the first block equation onto
	\(\operatorname{span}\{\one_C\}\) and \(\one_C^\perp\), respectively,
	gives
	\(
		s
		:=
		\lambda-(N-1)
		=
		\frac{1}{N}b^{\mathsf T}y
	\)
	and
	\(
		z
		=
		\frac{1}{\lambda+1}\mathcal P_C Ry.
	\)
	Substituting the latter identity into the second block equation yields
	\begin{equation}
		\left(
		\lambda I_t
		-A_X
		-\frac{1}{\lambda+1}
		R^{\mathsf T}\mathcal P_C R
		\right)y
		=
		b.
		\label{eq:resolvent}
	\end{equation}

	Set
	\(
		S:=R^{\mathsf T}\mathcal P_C R,
		\;
		T:=A_X+(\lambda+1)^{-1}S.
	\)
	Let \(K_0:=\max\{K,1\}\). Replacing \(K\) by \(K_0\) preserves all
	the hypotheses and only enlarges the constants in the uniform
	remainder estimates. We may therefore assume that \(K\ge1\).

	If \(N\ge5K+2\), then the preceding Weyl bound gives
	\(
		\lambda
		\ge
		N-1-K
		\ge
		4K+1.
	\)
	Moreover, since \(\|\mathcal P_C\|=1\),
	\(
		\|S\|
		\le
		\|R\|^2
		\le
		K^2,
	\)
	and hence
	\(
		\|T\|
		\le
		K+\frac{K^2}{4K+2}
		<
		\frac54K.
	\)
	Consequently,
	\(
		\|\lambda^{-1}T\|
		<
		\frac{5K/4}{4K+1}
		<
		\frac12.
	\)
	Thus both the spectral separation and the Neumann expansion below
	hold uniformly whenever \(N\ge5K+2\).

	Equation~\eqref{eq:resolvent} and the Neumann series give
	\(
		y
		=
		\lambda^{-1}b
		+\lambda^{-2}Tb
		+\lambda^{-3}T^2b
		+\mathcal R_3,
	\)
	where
	\(
		\mathcal R_3
		=
		\lambda^{-1}
		\sum_{j=3}^{\infty}(\lambda^{-1}T)^j b.
	\)
	Therefore,
	\(
		\|\mathcal R_3\|
		\le
		2\|b\|\,\lambda^{-4}\|T\|^3
		=
		O_{t,K}(N^{-4})
	\)
	uniformly.

	Next,
	\(
		\lambda^{-2}Tb
		=
		\lambda^{-2}A_Xb
		+
		\lambda^{-2}(\lambda+1)^{-1}Sb,
	\)
	whereas
	\begin{align*}
		\lambda^{-3}T^2b
		={}&
		\lambda^{-3}A_X^2b\\
		&+
		\lambda^{-3}(\lambda+1)^{-1}
		(A_XS+SA_X)b\\
		&+
		\lambda^{-3}(\lambda+1)^{-2}S^2b\\
		={}&
		\lambda^{-3}A_X^2b
		+O_{t,K}(N^{-4}).
	\end{align*}
	It follows that
	\(
		y
		=
		\lambda^{-1}b
		+\lambda^{-2}A_Xb
		+\lambda^{-3}A_X^2b
		+\lambda^{-2}(\lambda+1)^{-1}Sb
		+O_{t,K}(N^{-4}).
	\)
	In particular,
	\(
		\|y\|=O_{t,K}(N^{-1}),
	\)
	and consequently
	\(
		s
		=
		\frac{1}{N}b^{\mathsf T}y
		=
		O_{t,K}(N^{-2}).
	\)
	Thus
	\(
		\lambda=N-1+O_{t,K}(N^{-2}).
	\)

	Writing \(\lambda=N-1+s\) and using
	\(s=O_{t,K}(N^{-2})\), we obtain
	\begin{align*}
		\lambda^{-1}
		&=(N-1)^{-1}+O_{t,K}(N^{-4}),\\
		\lambda^{-2}
		&=(N-1)^{-2}+O_{t,K}(N^{-5}),\\
		\lambda^{-3}
		&=(N-1)^{-3}+O_{t,K}(N^{-6}),\\
		(\lambda+1)^{-1}
		&=N^{-1}+O_{t,K}(N^{-4}).
	\end{align*}

	The \(O_{t,K}(N^{-4})\) remainder in \(y\) contributes
	\(O_{t,K}(N^{-5})\) to
	\(
	s=N^{-1}b^{\mathsf T}y.
	\)
	Since \(b\), \(A_X\), and \(S\) are uniformly bounded, replacing
	\(\lambda^{-j}\) by \((N-1)^{-j}\), for \(1\le j\le3\), and replacing
	\((\lambda+1)^{-1}\) by \(N^{-1}\) introduces an additional error of
	order \(O_{t,K}(N^{-5})\) in \(s\). Hence
	\[
	s
	=
	\frac{b^{\mathsf T}b}{N(N-1)}
	+
	\frac{b^{\mathsf T}A_Xb}{N(N-1)^2}
	+
	\frac{b^{\mathsf T}A_X^2b}{N(N-1)^3}
	+
	\frac{b^{\mathsf T}Sb}{N^2(N-1)^2}
	+
	O_{t,K}(N^{-5}).
	\]
	Recalling that
	\(
		S=R^{\mathsf T}\mathcal P_C R
	\)
	and \(\lambda=N-1+s\), we obtain
	\eqref{eq:fourth-order}. The last two displayed correction terms are
	both \(O_{t,K}(N^{-4})\); absorbing them into the remainder yields
	\eqref{eq:large-clique-expansion}.
\end{proof}
	
	\begin{corollary}
	\label{cor:power-series-form}
	Under the hypotheses of \Cref{lem:large-clique}, define
	\(
		\gamma_2:=b^{\mathsf T}b,
		\;
		\gamma_3:=b^{\mathsf T}A_Xb,
		\;
		\gamma_4:=b^{\mathsf T}A_X^2b,
		\;
		\eta:=b^{\mathsf T}R^{\mathsf T}Rb.
	\)
	For \(v\in C\), set
	\(
		N_X(v):=N_{G_N}(v)\cap X
	\)
	and define the \(b\)-weighted core load at \(v\) by
	\(
		\Lambda(v)
		:=
		\sum_{x\in N_X(v)}b(x).
	\)
	Then
	\(
		(Rb)_v=\Lambda(v)
		\;  (v\in C),
	\)
	and consequently
	\(
		\eta
		=
		\|Rb\|_2^2
		=
		\sum_{v\in C}\Lambda(v)^2.
	\)
	Moreover, uniformly under the hypotheses of
	\Cref{lem:large-clique},
	\begin{equation}
		\label{eq:pure-power-expansion}
		\begin{aligned}
			\rho(G_N)
			={}&
			N-1
			+\frac{\gamma_2}{N^2}
			+\frac{\gamma_2+\gamma_3}{N^3}\\
			&\quad
			+\frac{\gamma_2+2\gamma_3+\gamma_4+\eta}{N^4}
			+O_{t,K}(N^{-5}).
		\end{aligned}
	\end{equation}
	Thus, if two candidates first differ in one of the displayed
	coefficients by a nonzero quantity independent of \(N\), then the
	sign of that first difference determines their spectral ordering for
	all sufficiently large \(N\). In particular, once \(A_X\) and \(b\)
	are fixed, the first placement-sensitive contribution is \(\eta\),
	which appears in the coefficient of \(N^{-4}\).
\end{corollary}

\begin{proof}
	For every \(v\in C\), the definition of \(R\) gives
	\(
		(Rb)_v
		=
		\sum_{x\in X}R_{vx}b(x)
		=
		\sum_{x\in N_X(v)}b(x)
		=
		\Lambda(v).
	\)
	Therefore,
	\(
		\eta
		=
		b^{\mathsf T}R^{\mathsf T}Rb
		=
		\|Rb\|_2^2
		=
		\sum_{v\in C}\Lambda(v)^2.
	\)

	Since
	\(
		\mathcal P_C
		=
		I_N-\frac{1}{N}\one_C\one_C^{\mathsf T}
	\)
	and \(R^{\mathsf T}\one_C=b\), we have
	\begin{align*}
		b^{\mathsf T}R^{\mathsf T}\mathcal P_C Rb
		&=
		b^{\mathsf T}R^{\mathsf T}Rb
		-\frac{1}{N}
		b^{\mathsf T}R^{\mathsf T}\one_C
		\one_C^{\mathsf T}Rb\\
		&=
		\eta-\frac{(b^{\mathsf T}b)^2}{N}\\
		&=
		\eta-\frac{\gamma_2^2}{N}.
	\end{align*}

	The hypotheses of \Cref{lem:large-clique} imply that
	\(\gamma_2,\gamma_3,\gamma_4\), and \(\eta\) are uniformly bounded.
	Furthermore,
	\begin{align*}
		\frac{1}{N(N-1)}
		&=
		\frac{1}{N^2}
		+\frac{1}{N^3}
		+\frac{1}{N^4}
		+O(N^{-5}),\\
		\frac{1}{N(N-1)^2}
		&=
		\frac{1}{N^3}
		+\frac{2}{N^4}
		+O(N^{-5}),\\
		\frac{1}{N(N-1)^3}
		&=
		\frac{1}{N^4}
		+O(N^{-5}),\\
		\frac{1}{N^2(N-1)^2}
		&=
		\frac{1}{N^4}
		+O(N^{-5}).
	\end{align*}
	Consequently,
	\[
		\frac{
			b^{\mathsf T}R^{\mathsf T}\mathcal P_C Rb
		}{
			N^2(N-1)^2
		}
		=
		\frac{\eta}{N^4}
		+O_{t,K}(N^{-5}).
	\]
	Substituting these expansions into
	\eqref{eq:fourth-order} and collecting equal powers of \(N^{-1}\)
	gives \eqref{eq:pure-power-expansion}.

	The final comparison assertion follows from the uniform
	\(O_{t,K}(N^{-5})\) remainder: the first nonzero difference among the
	coefficients of \(N^{-2}\), \(N^{-3}\), and \(N^{-4}\) dominates all
	subsequent terms for sufficiently large \(N\).
\end{proof}

	\begin{corollary}[Uniform comparison by the first nonzero coefficient]
		\label{cor:uniform-lexicographic-comparison}
		Let \((G_N)\) and \((H_N)\) be two sequences of graphs satisfying the
		hypotheses of \Cref{lem:large-clique} uniformly in \(N\). Suppose that,
		for some integer \(j\ge 1\),
		\(
		\rho(G_N)-\rho(H_N)
		=
		\Delta_N N^{-j}+O(N^{-j-1})
		\)
		uniformly in \(N\), where either
		\(\Delta_N\ge\varepsilon>0\) for all sufficiently large \(N\), or
		\(\Delta_N\le-\varepsilon<0\) for all sufficiently large \(N\).
		Then, for all sufficiently large \(N\), the sign of
		\(\rho(G_N)-\rho(H_N)\) agrees with the sign of \(\Delta_N\).
	\end{corollary}
	
	\begin{proof}
		By uniformity, there exists a constant \(C>0\), independent of \(N\),
		such that
		\[
		\left|
		\rho(G_N)-\rho(H_N)-\Delta_NN^{-j}
		\right|
		\le
		CN^{-j-1}
		\]
		for all sufficiently large \(N\). Since
		\(
		|\Delta_N|N^{-j}\ge \varepsilon N^{-j},
		\)
		the remainder has absolute value less than half the absolute value of
		the leading term whenever
		\(
		N>\frac{2C}{\varepsilon}.
		\)
		Hence the leading term determines the sign of
		\(\rho(G_N)-\rho(H_N)\).
	\end{proof}
	
	\subsection{Maximum spectral radius}
\label{subsec:missing-edge-selection}

The proof of the maximum result proceeds in three stages. First,
\Cref{prop:max-square} determines the missing-edge graphs that maximize
the degree-square sum. Second, \Cref{clm:h3-tie} resolves the only tie,
between the star and triangle configurations, which occurs when \(h=3\).
Third, once the star has been identified as the missing-edge graph,
\Cref{thm:nested} shows that the weighted loads on the \(Y\)-side are
uniquely maximized by a nested placement.

\begin{lemma}[The boundary cases \(h=0\) and \(h=1\)]
	\label{lem:boundary-h}
	If \(h=0\), then
	\(\mathcal{L}_{\mathcal{H}}^{2}(n,k,\delta)\) consists of a single
	isomorphism class.
	If \(h=1\), then, for all sufficiently large \(n\), every graph
	attaining the maximum spectral radius in
	\(\mathcal{L}_{\mathcal{H}}^{2}(n,k,\delta)\) has the following
	structure: if \(uv\) is the unique edge of the missing-edge graph
	\(M\), then
	\(
		N_Y(u)=N_Y(v)=\{y\}
	\)
	for some \(y\in Y\).
\end{lemma}

\begin{proof}
	If \(h=0\), then
	\(
	E(M)=E(B)=\varnothing.
	\)
	Consequently, the prescribed sizes of \(X\), \(W\), and \(Y\),
	together with the defining adjacencies of the construction, determine
	a unique graph up to isomorphism.
	
	Now suppose that \(h=1\), and let \(uv\) be the unique edge of \(M\).
	By \Cref{prop:LH2-reparametrization},
	\(
		d_B(u)=d_B(v)=1,
		\;
		d_B(x)=0
		\; \text{for every }x\in X\setminus\{u,v\}.
	\)
	Thus, up to relabeling the vertices of \(Y\), there are only two
	possible placements: the two \(X\)--\(Y\) edges either have a common
	endpoint in \(Y\) or have distinct endpoints in \(Y\).
	
	The matrix \(A_M\) and the vector \(b_M\) are identical in the two
	placements. Hence, by \eqref{eq:fourth-order}, the first coefficient
	that can distinguish them is
	\(
		b_M^{\mathsf T}R^{\mathsf T}\mathcal P_C Rb_M.
	\)
	For either placement, define
	\[
		\Lambda_R(z)
		=
		(Rb_M)_z
		=
		\sum_{x\in N_X(z)}b_M(x),
		\qquad z\in C.
	\]
	Since
	\[
		\mathcal P_C
		=
		I_N-\frac{1}{N}\one_C\one_C^{\mathsf T}
		\qquad\text{and}\qquad
		R^{\mathsf T}\one_C=b_M,
	\]
	we have
	\[
		b_M^{\mathsf T}R^{\mathsf T}\mathcal P_C Rb_M
		=
		\sum_{z\in C}\Lambda_R(z)^2
		-
		\frac{1}{N}
		\bigl(b_M^{\mathsf T}b_M\bigr)^2.
	\]
	The correction term is fixed because \(b_M\) is the same in both
	placements. Moreover, every vertex of \(W\) is adjacent to all
	vertices of \(X\), so its weighted load is
	\(
		\sum_{x\in X}b_M(x),
	\)
	which is also independent of the placement.
	
	Put
	\(
		\beta=b_M(u)=b_M(v)=k-1.
	\)
	In the distinct-endpoint placement, the two nonzero \(Y\)-side loads
	are \(\beta\) and \(\beta\), and hence their contribution to the
	squared-load sum is
	\(
		\beta^2+\beta^2=2\beta^2.
	\)
	In the common-endpoint placement, the unique nonzero \(Y\)-side load
	is \(2\beta\), and its contribution is
	\(
		(2\beta)^2=4\beta^2.
	\)
	Therefore,
	\(
		b_M^{\mathsf T}R_{\mathrm{com}}^{\mathsf T}
		\mathcal P_C R_{\mathrm{com}}b_M
		-
		b_M^{\mathsf T}R_{\mathrm{dist}}^{\mathsf T}
		\mathcal P_C R_{\mathrm{dist}}b_M
		=
		2\beta^2.
	\)
	All preceding coefficients in \eqref{eq:fourth-order} agree for the
	two placements. Subtracting their uniform expansions therefore gives
	\[
		\rho(G_{\mathrm{com}})
		-
		\rho(G_{\mathrm{dist}})
		=
		\frac{2\beta^2}{N^2(N-1)^2}
		+
		O_k(N^{-5}).
	\]
	Since \(\beta=k-1>0\), the leading term is positive and has order
	\(N^{-4}\). Hence
	\Cref{cor:uniform-lexicographic-comparison} implies that, for all
	sufficiently large \(n\), the common-endpoint placement has strictly
	larger spectral radius. As these are the only two placement types, it
	is the unique maximizing placement up to isomorphism.
\end{proof}
	
	\begin{proposition}
	\label{prop:max-square}
	Let \(h\ge 1\), and let \(M\) be a simple graph with exactly \(h\)
	edges. Then
	\(
		\sum_{x\in V(M)} d_M(x)^2
		\le
		h(h+1).
	\)
	Equality holds if and only if the edges of \(M\) are pairwise
	intersecting; equivalently, every pair of distinct edges of \(M\)
	has a common endpoint. Consequently, if equality holds, then the
	graph obtained from \(M\) by deleting all isolated vertices is
	isomorphic to \(K_{1,h}\), except that \(K_3\) is also possible when
	\(h=3\).
\end{proposition}

\begin{proof}
	Let \(p(M)\) denote the number of unordered pairs of distinct edges
	of \(M\) having a common endpoint. Since \(M\) is simple, every pair
	of distinct edges has at most one common endpoint. Hence
	\(
		p(M)
		=
		\sum_{x\in V(M)}\binom{d_M(x)}{2}.
	\)
	Using
	\(
		d_M(x)^2
		=
		d_M(x)+2\binom{d_M(x)}{2}
	\)
	and the degree-sum formula, we obtain
	\begin{align*}
		\sum_{x\in V(M)}d_M(x)^2
		&=
		\sum_{x\in V(M)}d_M(x)
		+
		2\sum_{x\in V(M)}\binom{d_M(x)}{2} \\
		&=
		2h+2p(M) \\
		&\le
		2h+2\binom{h}{2} \\
		&=
		h(h+1).
	\end{align*}
	Equality holds if and only if
	\(
	p(M)=\binom{h}{2},
	\)
	which is equivalent to every pair of distinct edges of \(M\)
	having a common endpoint.

	Assume henceforth that equality holds. If \(h=1\), then \(M\)
	consists of a single edge together with isolated vertices, and the
	conclusion follows immediately. Thus, assume that \(h\ge2\).

	Suppose first that \(M\) contains a triangle. Any edge having a
	common endpoint with each of the three edges of this triangle must
	itself be one of the three triangle edges. Since the edges of \(M\)
	are pairwise intersecting, \(M\) therefore consists of this triangle
	together with isolated vertices. In particular, \(h=3\).

	Now suppose that \(M\) is triangle-free. Choose two distinct edges
	\(uv\) and \(uw\) of \(M\), where \(v\ne w\). Every other edge must
	have a common endpoint with both \(uv\) and \(uw\). If such an edge
	does not contain \(u\), then it must be \(vw\), producing the
	triangle \(uvw\), a contradiction. Therefore every edge of \(M\)
	contains \(u\). Since \(M\) is simple and has \(h\) edges, deleting
	its isolated vertices yields \(K_{1,h}\).
\end{proof}
	
	\begin{claim}[Resolving the \(h=3\) tie]
	\label{clm:h3-tie}
	Fix \(k\ge4\), and set \(h=3\). Let
	\(
		M_{\star}\cong K_{1,3}\cup4K_1,
		\;
		M_{\triangle}\cong K_3\cup5K_1.
	\)
	Then there exists \(n_0=n_0(k)\) such that, for every
	\(n\ge n_0\) and every choice of admissible placements
	\(B_{\star}\) and \(B_{\triangle}\),
	\(
		\rho\bigl(G_3(M_{\star},B_{\star})\bigr)
		>
		\rho\bigl(G_3(M_{\triangle},B_{\triangle})\bigr).
	\)
\end{claim}

\begin{proof}
	Put
	\(
		N=n-8,
		\;
		a=k-4,
		\;
		A_M=A(K_X-M),
		\;
		b_M=a\one_X+d_M,
	\)
	where \(d_M=(d_M(x))_{x\in X}\). The degree sequences of
	\(M_{\star}\) and \(M_{\triangle}\) are
	\(
		(3,1,1,1,0,0,0,0)
		\; \text{and}\;
		(2,2,2,0,0,0,0,0),
	\)
	respectively. Thus, for both graphs,
	\(
		\sum_{x\in X}d_M(x)=6
		\; \text{and}\;
		\sum_{x\in X}d_M(x)^2=12.
	\)
	Since
	\[
		b_M^{\mathsf T}b_M
		=
		8a^2
		+
		2a\sum_{x\in X}d_M(x)
		+
		\sum_{x\in X}d_M(x)^2,
	\]
	we have
	\(
		b_{M_{\star}}^{\mathsf T}b_{M_{\star}}
		=
		b_{M_{\triangle}}^{\mathsf T}b_{M_{\triangle}}.
	\)
	Moreover,
	\(
		\sum_{x\in X}b_{M_{\star}}(x)
		=
		\sum_{x\in X}b_{M_{\triangle}}(x)
		=
		8a+6.
	\)

	For any missing-edge graph \(M\) on \(X\), we have
	\begin{align*}
		b_M^{\mathsf T}A_Mb_M
		&=
		b_M^{\mathsf T}
		\bigl(J_8-I_8-A(M)\bigr)b_M \\
		&=
		\left(\sum_{x\in X}b_M(x)\right)^2
		-
		\sum_{x\in X}b_M(x)^2
		-
		2\sum_{xy\in E(M)}b_M(x)b_M(y).
	\end{align*}
	For \(M_{\star}\), each edge joins the vertex whose
	\(b_{M_{\star}}\)-coordinate is \(a+3\) to a vertex whose
	coordinate is \(a+1\). Hence
	\[
		\sum_{xy\in E(M_{\star})}
		b_{M_{\star}}(x)b_{M_{\star}}(y)
		=
		3(a+3)(a+1).
	\]
	For \(M_{\triangle}\), each edge joins two vertices whose
	\(b_{M_{\triangle}}\)-coordinates are both \(a+2\). Therefore,
	\(
		\sum_{xy\in E(M_{\triangle})}
		b_{M_{\triangle}}(x)b_{M_{\triangle}}(y)
		=
		3(a+2)^2.
	\)
	The difference between these two quantities is
	\(
		3(a+2)^2-3(a+3)(a+1)=3.
	\)
	Since the remaining two terms in the expression for
	\(b_M^{\mathsf T}A_Mb_M\) are equal for the two candidates, it
	follows that
	\begin{align*}
		&b_{M_{\star}}^{\mathsf T}
		A_{M_{\star}}b_{M_{\star}}
		-
		b_{M_{\triangle}}^{\mathsf T}
		A_{M_{\triangle}}b_{M_{\triangle}} \\
		&\qquad=
		2\left(
		3(a+2)^2-3(a+3)(a+1)
		\right)
		=
		6.
	\end{align*}

	By \eqref{eq:large-clique-expansion},
	\[
		\rho\bigl(G_3(M,B)\bigr)
		=
		N-1
		+
		\frac{b_M^{\mathsf T}b_M}{N(N-1)}
		+
		\frac{b_M^{\mathsf T}A_Mb_M}{N(N-1)^2}
		+
		O_k(N^{-4}),
	\]
	where the remainder is uniform over all admissible placements \(B\).
	Consequently,
	\[
		\rho\bigl(G_3(M_{\star},B_{\star})\bigr)
		-
		\rho\bigl(G_3(M_{\triangle},B_{\triangle})\bigr)
		=
		\frac{6}{N(N-1)^2}
		+
		O_k(N^{-4}),
	\]
	uniformly over \(B_{\star}\) and \(B_{\triangle}\).

	The leading term is positive and has order \(N^{-3}\), whereas the
	uniform remainder has order \(N^{-4}\). Hence the difference is
	positive for all sufficiently large \(N\), with a threshold
	depending only on \(k\). Since \(N=n-8\), the result follows.
\end{proof}
	
	\begin{theorem}
	\label{thm:star-core}
	Fix integers \(h\ge2\) and \(k\ge h+1\), and set
	\(\delta=k+h\). For every \(n\ge n_{\mathrm{int}}(k,\delta)\),
	the missing-edge graph of every spectral-radius maximizer in
	\(\mathcal{L}_{\mathcal H}^{2}(n,k,\delta)\) is isomorphic to
	\(
		M^\star
		:=
		K_{1,h}\cup(h+1)K_1.
	\)
\end{theorem}

\begin{proof}
	Since \(k\ge h+1\), we have
	\(
	k\le\delta=k+h\le2k-1,
	\)
	so \((k,\delta)\) lies in the admissible parameter range.

	Let
	\(
		G=G_h(M,B)
		\in
		\mathcal{L}_{\mathcal H}^{2}(n,k,\delta),
	\)
	and put
	\(
		a=k-h-1,
		\;
		t=|X|=2h+2,
		\;
		N=n-t=n-2h-2.
	\)
	By \eqref{eq:bM},
	\(
	b_M=a\one_X+d_M,
	\)
	where \(d_M=(d_M(x))_{x\in X}\). Since \(e(M)=h\), the
	degree-sum formula gives
	\(
	\sum_{x\in X}d_M(x)=2h.
	\)
	Therefore,
	\begin{align}
		b_M^{\mathsf T}b_M
		&=
		\sum_{x\in X}\bigl(a+d_M(x)\bigr)^2 \notag\\
		&=
		ta^2+4ah+\sum_{x\in X}d_M(x)^2.
		\label{eq:bM-square-degree}
	\end{align}
	Thus, among the admissible missing-edge graphs \(M\), the first
	\(M\)-dependent term in \eqref{eq:large-clique-expansion} is
	maximized precisely when the sum of squared degrees
	\(
	\sum_{x\in X}d_M(x)^2
	\)
	is maximized.

	By \Cref{prop:max-square},
	\(
		\sum_{x\in X}d_M(x)^2
		\le
		h(h+1).
	\)
	When \(h=2\) or \(h\ge4\), equality holds only when the
	nonisolated part of \(M\) is isomorphic to \(K_{1,h}\). Since
	\(|X|=2h+2\), this is equivalent to
	\(
		M\cong K_{1,h}\cup(h+1)K_1=M^\star.
	\)

	Moreover, since \(r^2\equiv r\pmod 2\) for every integer \(r\),
	\(
		\sum_{x\in X}d_M(x)^2
		\equiv
		\sum_{x\in X}d_M(x)
		=
		2h
		\equiv0
		\pmod 2.
	\)
	The quantity \(h(h+1)\) is also even. Hence every strict deficit
	from \(h(h+1)\) is at least \(2\).

	Fix an admissible placement \(B^\star\) of \(M^\star\), and let
	\(
		G^\star=G_h(M^\star,B^\star).
	\)
	If \(M\) does not attain equality in \Cref{prop:max-square}, define
	\(
		\Delta(M)
		=
		h(h+1)-\sum_{x\in X}d_M(x)^2.
	\)
	Then \(\Delta(M)\ge2\). Subtracting the two instances of
	\eqref{eq:large-clique-expansion} yields
	\begin{align*}
		\rho(G^\star)-\rho(G)
		&=
		\frac{\Delta(M)}{N(N-1)}
		+
		O_{h,k}(N^{-3})\\
		&=
		\frac{\Delta(M)}{N^2}
		+
		O_{h,k}(N^{-3}),
	\end{align*}
	uniformly over all admissible choices of \(B^\star\) and \(B\).
	Since \(\Delta(M)\ge2\), the leading coefficient has a uniform
	positive lower bound. By
	\Cref{cor:uniform-lexicographic-comparison} and the definition of
	\(n_{\mathrm{int}}(k,\delta)\), the above difference is positive
	whenever \(n\ge n_{\mathrm{int}}(k,\delta)\). Consequently, no
	missing-edge graph with a strict squared-degree deficit can occur
	in a spectral-radius maximizer.

	It remains to consider the additional equality case in
	\Cref{prop:max-square}. This occurs only when \(h=3\), in which case
	\(
		M\cong K_3\cup5K_1.
	\)
	Here \(k\ge h+1=4\), so \Cref{clm:h3-tie} applies and resolves the
	tie in favor of
	\(
	K_{1,3}\cup4K_1.
	\)
	The comparison is uniform over all admissible placements, and its
	threshold is included in the definition of
	\(n_{\mathrm{int}}(k,\delta)\). Therefore, every spectral-radius
	maximizer has missing-edge graph isomorphic to \(M^\star\).
\end{proof}

   \begin{theorem}
	\label{thm:nested}
	Fix integers \(h\ge2\) and \(k\ge h+1\), set
	\(\delta=k+h\), and let
	\(
		M^\star
		:=
		K_{1,h}\cup(h+1)K_1.
	\)
	For every \(n\ge n_{\mathrm{int}}(k,\delta)\), let
	\(B^{\mathrm{nest}}\) be any nested admissible placement of
	\(M^\star\). Then, for every admissible placement \(B\),
	\(
		\rho\bigl(G_h(M^\star,B)\bigr)
		\le
		\rho\bigl(G_h(M^\star,B^{\mathrm{nest}})\bigr).
	\)
	Equality holds if and only if \(B\) is nested. Moreover, all nested
	placements are equivalent under a relabeling of the vertices of
	\(Y\). In particular, the nested placement is the unique
	spectral-radius-maximizing placement up to isomorphism.
\end{theorem}

\begin{proof}
	Let \(c\) be the center of the \(K_{1,h}\) component of
	\(M^\star\), let \(\ell_1,\ldots,\ell_h\) be its leaves, and let
	\(I\) be the set of its \(h+1\) isolated vertices. Put
	\(
		N=|C|=n-2h-2,
		\;
		b=b_{M^\star}.
	\)
	For an admissible placement \(B\), let \(R_B\) denote the
	\(C\)--\(X\) adjacency matrix of \(G_h(M^\star,B)\).

	Once \(M^\star\) is fixed, both
	\(
		A_{M^\star}=A(K_X-M^\star)
		\; \text{and}\;
		b=b_{M^\star}
	\)
	are independent of the placement. It follows from
	\eqref{eq:fourth-order} that the first term in the expansion that
	can distinguish two placements is
	\(
		b^{\mathsf T}R_B^{\mathsf T}\mathcal P_C R_Bb.
	\)

	For \(v\in C\), define
	\(
		\Lambda_B(v)
		=
		\sum_{x\in N_X(v)}b(x).
	\)
	Since
	\(
	\mathcal P_C
	=
		I_N-N^{-1}\one_C\one_C^{\mathsf T},
	\)
	we have
	\[
		b^{\mathsf T}R_B^{\mathsf T}\mathcal P_C R_Bb
		=
		\sum_{v\in C}\Lambda_B(v)^2
		-
		\frac{1}{N}
		\left(\sum_{v\in C}\Lambda_B(v)\right)^2.
	\]
	Moreover,
	\[
		\sum_{v\in C}\Lambda_B(v)
		=
		\one_C^{\mathsf T}R_Bb
		=
		\bigl(R_B^{\mathsf T}\one_C\bigr)^{\mathsf T}b
		=
		b^{\mathsf T}b,
	\]
	which is independent of \(B\). The contribution
	\(
	\sum_{w\in W}\Lambda_B(w)^2
	\)
	is also fixed, since every vertex of \(W\) is adjacent to every
	vertex of \(X\). Consequently, maximizing the placement-sensitive
	term is equivalent to maximizing
	\(
		\Phi(B)
		:=
		\sum_{y\in Y}\Lambda_B(y)^2.
	\)

	Put
	\(
		\alpha=b(c)=k-1,
		\;
		\beta=b(\ell_i)=k-h
		\; (1\le i\le h).
	\)
	Then \(\alpha>\beta>0\). Each vertex of \(I\) has no neighbor in
	\(Y\), because
	\(
		d_B(x)=d_{M^\star}(x)=0
		\; (x\in I).
	\)

	All \(Y\)-neighborhoods below are taken with respect to the current
	placement. For each \(y\in Y\), define
	\[
		\varepsilon_y
		=
		\begin{cases}
			1,&y\in N_Y(c),\\
			0,&y\notin N_Y(c),
		\end{cases}
		\qquad
		r_y
		=
		\bigl|\{i:y\in N_Y(\ell_i)\}\bigr|.
	\]
	Since \(d_B(c)=h\) and \(d_B(\ell_i)=1\), we have
	\(
		\sum_{y\in Y}\varepsilon_y=h,
		\;
		\sum_{y\in Y}r_y=h,
	\)
	and
	\(
		\Lambda_B(y)
		=
		\alpha\varepsilon_y+\beta r_y.
	\)

	We first show that a placement having positive leaf load outside
	\(N_Y(c)\) cannot maximize \(\Phi\). Suppose that
	\(y\notin N_Y(c)\) and \(r_y>0\). There exists
	\(z\in N_Y(c)\) such that \(r_z<r_y\). Indeed, otherwise,
	\[
		\sum_{u\in Y}r_u
		\ge
		\sum_{u\in N_Y(c)}r_u+r_y
		\ge
		hr_y+r_y
		\ge
		h+1,
	\]
	contrary to \(\sum_{u\in Y}r_u=h\).

	Replace the edge \(cz\) with \(cy\), and denote the resulting
	placement by \(B'\). This replacement preserves all prescribed
	\(X\)-side degrees and hence produces another admissible placement.
	Only the loads at \(y\) and \(z\) change, and
	\begin{align*}
		\Phi(B')-\Phi(B)
		&=
		(\alpha+\beta r_y)^2+(\beta r_z)^2\\
		&\quad
		-(\beta r_y)^2-(\alpha+\beta r_z)^2\\
		&=
		2\alpha\beta(r_y-r_z)\\
		&\ge
		2\alpha\beta.
	\end{align*}
	Thus every placement maximizing \(\Phi\) satisfies
	\(
		r_y>0
		\;\Longrightarrow\;
		y\in N_Y(c).
	\)

	It remains to distribute the \(h\) attachment edges incident with
	the leaves among the \(h\) vertices of \(N_Y(c)\). Suppose that two
	distinct vertices \(y,z\in N_Y(c)\) have positive leaf loads
	\(
		r_y=r\ge s=r_z>0.
	\)
	Choose a leaf \(\ell_i\) adjacent to \(z\), and replace
	\(\ell_i z\) with \(\ell_i y\). Since \(z\) is the unique
	\(Y\)-neighbor of \(\ell_i\), the edge \(\ell_i y\) is not already
	present. Thus the replacement is admissible. If \(B'\) denotes the
	resulting placement, then
	\begin{align*}
		\Phi(B')-\Phi(B)
		&=
		(\alpha+\beta(r+1))^2
		+
		(\alpha+\beta(s-1))^2\\
		&\quad
		-
		(\alpha+\beta r)^2
		-
		(\alpha+\beta s)^2\\
		&=
		2\beta^2(r-s+1)\\
		&\ge
		2\beta^2.
	\end{align*}
	Therefore, in every placement maximizing \(\Phi\), all \(h\)
	attachment edges incident with the leaves have a common endpoint
	in \(N_Y(c)\). By the definition of a nested placement, this is
	precisely the nested configuration.

	More quantitatively, the two preceding edge replacements show that
	every non-nested placement \(B\) admits an admissible placement
	\(B'\) such that
	\(
		\Phi(B')-\Phi(B)
		\ge
		\min\{2\alpha\beta,2\beta^2\}
		=
		2\beta^2.
	\)
	Since the total-load correction and the contribution from \(W\)
	are fixed, it follows that
	\(
		b^{\mathsf T}R_{B'}^{\mathsf T}\mathcal P_C R_{B'}b
		-
		b^{\mathsf T}R_B^{\mathsf T}\mathcal P_C R_Bb
		=
		\Phi(B')-\Phi(B).
	\)
	Using \eqref{eq:fourth-order}, we obtain
	\[
		\rho\bigl(G_h(M^\star,B')\bigr)
		-
		\rho\bigl(G_h(M^\star,B)\bigr)
		=
		\frac{\Phi(B')-\Phi(B)}
		{N^2(N-1)^2}
		+
		O_{h,k}(N^{-5}),
	\]
	where the remainder is uniform over all admissible placements.

	Since \(\Phi(B')-\Phi(B)\ge2\beta^2\) and
	\(\beta=k-h\ge1\), the leading term is positive and has order
	\(N^{-4}\), whereas the uniform remainder has order \(N^{-5}\).
	By \Cref{cor:uniform-lexicographic-comparison} and the definition of
	\(n_{\mathrm{int}}(k,\delta)\), every such admissible replacement
	strictly increases the spectral radius whenever
	\(n\ge n_{\mathrm{int}}(k,\delta)\).

	Starting from any non-nested placement and repeatedly applying one
	of the preceding replacements therefore produces a strictly
	increasing sequence of values of both \(\Phi\) and the spectral
	radius. The process terminates because, for fixed \(n\), there are
	only finitely many admissible placements. At termination, no leaf
	load lies outside \(N_Y(c)\), and at most one vertex of \(N_Y(c)\)
	has positive leaf load. Since the total leaf load is \(h\), the
	terminal placement is nested. Hence every non-nested placement has
	strictly smaller spectral radius than a nested placement.

	Finally, all nested placements are equivalent under a relabeling
	of the vertices of \(Y\). Therefore they yield isomorphic graphs
	and have the same spectral radius. This proves the asserted
	maximality, equality characterization, and uniqueness up to
	isomorphism.
\end{proof}

\begin{proof}[Proof of \Cref{thm:intro-ordinary}]
	The cases \(h=0\) and \(h=1\) follow from
	\Cref{lem:boundary-h}. Now assume \(h\ge2\).
	By \Cref{thm:star-core}, every spectral-radius maximizer has
	missing-edge graph isomorphic to
	\(
		M^\star=K_{1,h}\cup(h+1)K_1.
	\)
	Once the missing-edge graph is fixed, \Cref{thm:nested} shows that
	the placement must be nested. Since all nested placements are
	equivalent up to isomorphism, \(G^\star_{h,k,n}\) is the unique
	maximizer up to isomorphism, as claimed.
\end{proof}
	
	\subsection{Minimum spectral radius}
\label{subsec:spectral-minimum-exceptional-family}

The large-clique expansion also determines the minimum spectral radius
within
\(\mathcal{L}_{\mathcal H}^{2}(n,k,\delta)\). In contrast to the
maximizer, which concentrates both the degrees of the missing-edge graph
and the \(Y\)-side loads, the minimizer disperses both quantities as
evenly as permitted by the defining constraints. This
concentration--dispersion dichotomy determines the two endpoints of the
spectral range within the exceptional family.

The proof again separates core selection from placement selection. Since
\(d^2\ge d\) for every nonnegative integer \(d\), we have
\[
	\sum_{x\in X}d_M(x)^2
	\ge
	\sum_{x\in X}d_M(x)
	=
	2h.
\]
Equality holds if and only if every vertex of \(M\) has degree at most
one, which selects the matching
\(
hK_2\cup2K_1.
\)
Once this missing-edge graph is fixed, minimizing the \(Y\)-side
squared-load sum forces the \(2h\) exceptional \(X\)--\(Y\) edges to
have pairwise distinct endpoints in \(Y\).

When \(h=0\), let \(G^{\min}_{0,k,n}\) denote the unique isomorphism
type in \(\mathcal{L}_{\mathcal H}^{2}(n,k,k)\). Now assume that
\(h\ge1\). We define
\(
	G^{\min}_{h,k,n}
	=
	G_h\bigl(M^{\min},B^{\mathrm{disp}}\bigr)
\)
as follows. Write
\(
	X=U\mathbin{\dot\cup}I,
	\;
	U=\{u_1,v_1,\ldots,u_h,v_h\},
	\;
	|I|=2,
\)
and let \(M^{\min}\) be the missing-edge graph on \(X\) defined by
\(
	E(M^{\min})
	=
	\{u_iv_i:1\le i\le h\}.
\)
Thus
\(
	M^{\min}\cong hK_2\cup2K_1.
\)

Under the standing assumptions \(n\ge2\delta+2\) and \(k\ge h+1\),
\begin{align*}
	|Y|
	&=
	n-k-h-1\\
	&=
	n-\delta-1\\
	&\ge
	\delta+1\\
	&=
	k+h+1\\
	&\ge
	2h+2.
\end{align*}
We may therefore choose pairwise distinct vertices
\(
y_1,\ldots,y_{2h}\in Y.
\)
Define the dispersed placement \(B^{\mathrm{disp}}\) by
\begin{equation}
	\label{eq:minimum-dispersed-placement}
	N_Y(u_i)=\{y_{2i-1}\},
	\qquad
	N_Y(v_i)=\{y_{2i}\},
	\qquad
	1\le i\le h,
\end{equation}
and set
\(
N_Y(x)=\varnothing
\)
for every \(x\in I\). There are no other edges between \(X\) and \(Y\).

For every \(x\in X\), the construction gives
\(
	d_{B^{\mathrm{disp}}}(x)
	=
	d_{M^{\min}}(x).
\)
Indeed, both degrees are equal to \(1\) for \(x\in U\) and to \(0\)
for \(x\in I\). Moreover,
\(
	e(M^{\min})=h,
	\;
	e(B^{\mathrm{disp}})=2h.
\)
Hence \((M^{\min},B^{\mathrm{disp}})\) is an admissible pair by
\Cref{prop:LH2-reparametrization}. All remaining adjacencies are
determined by the parametrization:
\(
	G[X]=K_X-M^{\min},
\)
the set \(W\cup Y\) induces a clique, and \(W\) is complete to \(X\).

Any two choices of the matching labels and of the \(2h\) distinct
vertices of \(Y\) are related by permutations of \(X\), \(W\), and
\(Y\) that preserve the three parts. Consequently,
\(G^{\min}_{h,k,n}\) is well-defined up to isomorphism.
	
	\begin{proof}[Proof of \Cref{thm:intro-ordinary-minimum}]
	If \(h=0\), then \(E(M)=E(B)=\varnothing\), and the family contains
	a single isomorphism type. Hence the assertion is immediate. Assume
	henceforth that \(h\ge1\).

	We use the notation of
	\Cref{prop:LH2-reparametrization,lem:large-clique}. In particular,
	\(
		|X|=t=2h+2,
		\;
		a=k-h-1,
		\;
		C=W\cup Y,
		\;
		N=|C|=n-2h-2.
	\)
	Every graph in
	\(\mathcal{L}_{\mathcal H}^{2}(n,k,\delta)\) can be written as
	\(G_h(M,B)\), where \(e(M)=h\) and
	\(
		d_B(x)=d_M(x)
		\; (x\in X).
	\)
	For an admissible placement \(B\), let \(R_B\) be the
	\(N\times t\) adjacency matrix between \(C\) and \(X\), and put
	\(
		A_M=A(K_X-M).
	\)
	By \eqref{eq:bM},
	\(
		b_M
		=
		R_B^{\mathsf T}\one_C
		=
		a\one_X+d_M,
	\)
	where \(d_M=(d_M(x))_{x\in X}\).

	We first determine the missing-edge graph. Since
	\(
	\sum_{x\in X}d_M(x)=2h,
	\)
	we have
	\begin{align}
		b_M^{\mathsf T}b_M
		&=
		\sum_{x\in X}\bigl(a+d_M(x)\bigr)^2 \notag\\
		&=
		ta^2+4ah+\sum_{x\in X}d_M(x)^2.
		\label{eq:minimum-leading-objective}
	\end{align}
	For every nonnegative integer \(r\), we have \(r^2\ge r\), with
	equality if and only if \(r\in\{0,1\}\). Consequently,
	\begin{equation}
		\sum_{x\in X}d_M(x)^2
		\ge
		\sum_{x\in X}d_M(x)
		=
		2h.
		\label{eq:minimum-square-degree-bound}
	\end{equation}
	Equality holds if and only if
	\(
		d_M(x)\in\{0,1\}
		\; (x\in X).
	\)
	Since \(M\) has \(h\) edges on \(2h+2\) vertices, this is
	equivalent to
	\(
		M\cong hK_2\cup2K_1.
	\)

	If \(M\) is not a matching, define
	\(
		\Delta_M
		:=
		\sum_{x\in X}d_M(x)^2-2h
		=
		\sum_{x\in X}d_M(x)\bigl(d_M(x)-1\bigr).
	\)
	Then \(\Delta_M\ge2\), because at least one vertex has degree at
	least \(2\).

	Let
	\(
		M^{\min}:=hK_2\cup2K_1,
	\)
	and fix any admissible placement \(B_0\) of \(M^{\min}\).
	Subtracting the two instances of
	\eqref{eq:large-clique-expansion} gives
	\begin{align*}
		&\rho\bigl(G_h(M,B)\bigr)
		-
		\rho\bigl(G_h(M^{\min},B_0)\bigr)\\
		&\qquad=
		\frac{\Delta_M}{N(N-1)}
		+
		O_{h,k}(N^{-3})\\
		&\qquad=
		\frac{\Delta_M}{N^2}
		+
		O_{h,k}(N^{-3}),
	\end{align*}
	uniformly over all admissible \(M\), \(B\), and \(B_0\).
	Since \(\Delta_M\ge2\), the leading coefficient has a uniform
	positive lower bound. By
	\Cref{cor:uniform-lexicographic-comparison} and the definition of
	\(n_{\mathrm{int}}(k,\delta)\), the above difference is positive
	for every \(n\ge n_{\mathrm{int}}(k,\delta)\). Thus the
	missing-edge graph of every spectral-radius minimizer is isomorphic
	to
	\(
		M^{\min}=hK_2\cup2K_1.
	\)

	It remains to minimize over the admissible placements of
	\(M^{\min}\). Let
	\(
		U
		=
		\{x\in X:d_{M^{\min}}(x)=1\}
	\)
	be the set of endpoints of the matching edges, and put
	\(
	I=X\setminus U.
	\)
	Then
	\(
		|U|=2h,
		\;
		|I|=2.
	\)
	Set
	\(
		\beta=a+1=k-h>0.
	\)
	For \(x\in X\),
	\[
		b_{M^{\min}}(x)
		=
		\begin{cases}
			\beta,&x\in U,\\
			a,&x\in I.
		\end{cases}
	\]
	Moreover, every admissible placement \(B\) satisfies
	\[
		d_B(x)
		=
		\begin{cases}
			1,&x\in U,\\
			0,&x\in I.
		\end{cases}
	\]
	Thus every vertex of \(U\) has exactly one neighbor in \(Y\), while
	the two vertices of \(I\) have no neighbors in \(Y\).

	For \(y\in Y\), define
	\(
		r_y
		=
		|N_B(y)|
		=
		|N_G(y)\cap U|.
	\)
	The prescribed \(X\)-side degrees give
	\begin{equation}
		r_y\in\mathbb Z_{\ge0},
		\qquad
		\sum_{y\in Y}r_y=2h.
		\label{eq:minimum-load-constraints}
	\end{equation}

	Put \(b=b_{M^{\min}}\). For \(v\in C\), define
	\(
		\Lambda_B(v)
		=
		\sum_{x\in N_X(v)}b(x).
	\)
	Once \(M^{\min}\) is fixed, both \(A_{M^{\min}}\) and \(b\) are
	fixed. By \eqref{eq:fourth-order}, the first
	placement-sensitive quantity is
	\(
		b^{\mathsf T}
		R_B^{\mathsf T}\mathcal P_C R_Bb.
	\)
	It satisfies
	\[
		b^{\mathsf T}R_B^{\mathsf T}\mathcal P_C R_Bb
		=
		\sum_{v\in C}\Lambda_B(v)^2
		\quad-
		\frac{1}{N}
		\left(\sum_{v\in C}\Lambda_B(v)\right)^2.
	\]
	The total load is independent of \(B\), because
	\(
		\sum_{v\in C}\Lambda_B(v)
		=
		\one_C^{\mathsf T}R_Bb
		=
		\bigl(R_B^{\mathsf T}\one_C\bigr)^{\mathsf T}b
		=
		b^{\mathsf T}b.
	\)
	The contribution from \(W\) is also independent of the placement,
	since every vertex of \(W\) is adjacent to every vertex of \(X\).

	For \(y\in Y\), all neighbors of \(y\) in \(X\) lie in \(U\), and
	each has \(b\)-weight \(\beta\). Therefore,
	\(
		\Lambda_B(y)=\beta r_y.
	\)
	It follows that minimizing the placement-sensitive quantity is
	equivalent to minimizing
	\(
		\sum_{y\in Y}\Lambda_B(y)^2
		=
		\beta^2\sum_{y\in Y}r_y^2.
	\)
	By \eqref{eq:minimum-load-constraints},
	\begin{equation}
		\sum_{y\in Y}r_y^2
		\ge
		\sum_{y\in Y}r_y
		=
		2h,
		\label{eq:minimum-load-square-bound}
	\end{equation}
	with equality if and only if
	\(
		r_y\in\{0,1\}
		\; (y\in Y).
	\)

	Since
	\(
	n\ge n_{\mathrm{int}}(k,\delta)\ge2\delta+2
	\)
	and \(k\ge h+1\), we have
	\(
		|Y|
		=
		n-k-h-1
		\ge
		k+h+1
		\ge
		2h+2.
	\)
	Hence equality in \eqref{eq:minimum-load-square-bound} is
	attainable. It holds precisely when the \(2h\) exceptional
	\(X\)--\(Y\) edges have pairwise distinct \(Y\)-endpoints, that is,
	when the placement is dispersed.

	If \(B\) is not dispersed, define
	\(
		\Delta_B
		:=
		\sum_{y\in Y}r_y^2-2h
		=
		\sum_{y\in Y}r_y(r_y-1).
	\)
	Then \(\Delta_B\ge2\). Let \(B^{\mathrm{disp}}\) be a dispersed
	placement, and let \(R_{B^{\mathrm{disp}}}\) be its corresponding
	\(C\)-by-\(X\) adjacency matrix. Since the total-load correction
	and the contribution from \(W\) are fixed,
	\(
		b^{\mathsf T}R_B^{\mathsf T}\mathcal P_C R_Bb
		-
		b^{\mathsf T}
		R_{B^{\mathrm{disp}}}^{\mathsf T}
		\mathcal P_C
		R_{B^{\mathrm{disp}}}b
		=
		\beta^2\Delta_B.
	\)
	All terms preceding the placement-sensitive term in
	\eqref{eq:fourth-order} agree. Hence, uniformly over all admissible
	placements,
	\begin{align*}
		&\rho\bigl(G_h(M^{\min},B)\bigr)
		-
		\rho\bigl(G_h(M^{\min},B^{\mathrm{disp}})\bigr)\\
		&\qquad=
		\frac{\beta^2\Delta_B}{N^2(N-1)^2}
		+
		O_{h,k}(N^{-5})\\
		&\qquad=
		\frac{\beta^2\Delta_B}{N^4}
		+
		O_{h,k}(N^{-5}).
	\end{align*}
	The leading coefficient is at least
	\(
	2\beta^2>0.
	\)
	Therefore, by
	\Cref{cor:uniform-lexicographic-comparison} and the definition of
	\(n_{\mathrm{int}}(k,\delta)\), every non-dispersed placement has
	strictly larger spectral radius whenever
	\(n\ge n_{\mathrm{int}}(k,\delta)\). Thus every minimizer has a
	dispersed placement.

	Finally, all copies of \(hK_2\cup2K_1\) on \(X\) are equivalent
	under a relabeling of \(X\). For a fixed copy of \(M^{\min}\), any
	two dispersed placements are equivalent under a relabeling of
	\(Y\). Hence all graphs obtained from the matching core with a
	dispersed placement are mutually isomorphic. The preceding strict
	comparisons exclude every other missing-edge graph and every other
	placement, proving both the asserted minimum and its uniqueness up
	to isomorphism.
\end{proof}

	\subsubsection{A common internal threshold}
\label{subsec:common-threshold}

All asymptotic estimates used above were obtained before a common
threshold was specified. We now define such a threshold and thereby make
every preceding occurrence of ``sufficiently large'' precise.

Fix integers \(k\ge2\) and \(\delta\) satisfying
\(
	k\le\delta\le2k-1,
\)
put \(h=\delta-k\), and let
\(
	N=n-2h-2.
\)
The preceding proofs reduce to six types of pairwise spectral
comparisons. The first four establish the maximizing structure:
\begin{enumerate}
	\item when \(h=1\), the common-endpoint placement has larger spectral
	radius than the distinct-endpoint placement;
	\item the star core has larger spectral radius than any core with a
	strictly smaller degree-square sum;
	\item when \(h=3\), the star core has larger spectral radius than the
	triangle core;
	\item among placements of the star core, the nested placement has
	larger spectral radius than any non-nested placement.
\end{enumerate}
The remaining two establish the minimizing structure:
\begin{enumerate}
	\setcounter{enumi}{4}
	\item every nonmatching core has larger spectral radius than the
	matching core;
	\item among placements of the matching core, every non-dispersed
	placement has larger spectral radius than the dispersed placement.
\end{enumerate}
Only the applicable comparison types are included. Let
\(\mathfrak C_{h,k}\) denote the set of all directed comparisons of
these types, taken up to relabeling of \(X\) and of the active vertices
of \(Y\). By \Cref{rem:finite-patterns}, the set
\(\mathfrak C_{h,k}\) is finite.

For each \(\mathfrak c\in\mathfrak C_{h,k}\), let
\(D_{\mathfrak c}(N)\) denote the corresponding oriented spectral
difference. The uniform estimates established above give constants
\(
	j_{\mathfrak c}\in\{2,3,4\},
	\;
	\varepsilon_{\mathfrak c}>0,
	\;
	C_{\mathfrak c}\ge0,
	\;
	N_{\mathfrak c}\in\mathbb N,
\)
such that
\(
	D_{\mathfrak c}(N)
	\ge
	\varepsilon_{\mathfrak c}N^{-j_{\mathfrak c}}
	-
	C_{\mathfrak c}N^{-j_{\mathfrak c}-1}
	\;
	(N\ge N_{\mathfrak c}).
\)
Here \(C_{\mathfrak c}\) and \(N_{\mathfrak c}\) are independent of
\(N\) and of the labeled realization of the corresponding comparison.
Since \(h\) and \(k\) are fixed and
\(\mathfrak C_{h,k}\) is finite, all these constants ultimately depend
only on \(h\) and \(k\). We choose each \(N_{\mathfrak c}\) sufficiently
large that the relevant instance of \Cref{lem:large-clique}, together
with its uniform remainder estimate, applies.

In the order listed above, uniform positive lower bounds for the leading
coefficients may be chosen as
\(
	2(k-1)^2,
	\;
	2,
	\;
	6,
	\;
	2(k-h)^2,
	\;
	2,
	\;
	2(k-h)^2,
\)
with the inapplicable cases omitted. Define
\[
	N_{\mathrm{int}}(h,k)
	:=
	\max_{\mathfrak c\in\mathfrak C_{h,k}}
	\left\{
		N_{\mathfrak c},
		\left\lfloor
		\frac{2C_{\mathfrak c}}
		{\varepsilon_{\mathfrak c}}
		\right\rfloor+1
	\right\},
\]
where the maximum over the empty set is understood to be \(0\), and set
\begin{equation}
	\label{eq:rigorous-nint}
	n_{\mathrm{int}}(k,\delta)
	:=
	\max
	\left\{
		2\delta+2,\,
		2h+2+N_{\mathrm{int}}(h,k)
	\right\}.
\end{equation}

Indeed, if \(n\ge n_{\mathrm{int}}(k,\delta)\), then
\(N\ge N_{\mathrm{int}}(h,k)\). Hence, for every
\(\mathfrak c\in\mathfrak C_{h,k}\),
\(
	N>
	\frac{2C_{\mathfrak c}}{\varepsilon_{\mathfrak c}},
\)
and therefore
\[
	D_{\mathfrak c}(N)
	\ge
	N^{-j_{\mathfrak c}}
	\left(
	\varepsilon_{\mathfrak c}
	-
	\frac{C_{\mathfrak c}}{N}
	\right)
	>
	\frac{\varepsilon_{\mathfrak c}}{2}
	N^{-j_{\mathfrak c}}
	>
	0.
\]
Thus every required directed spectral difference has the asserted
strict sign whenever \(n\ge n_{\mathrm{int}}(k,\delta)\). This proves
the existence of a common internal threshold using only the finite
collection of uniform remainder bounds and positive leading gaps,
without defining the threshold in terms of the conclusions of either
extremal theorem.

\begin{remark}[Concentration versus dispersion]
	\label{rem:concentration-versus-dispersion}
	For \(h\ge1\), the two endpoints of the spectral range are governed
	by opposite placement principles: the spectral maximizer
	concentrates the weighted loads on the \(Y\)-side, whereas the
	spectral minimizer disperses them.

	For \(h\ge2\), the two extremizers also have different missing-edge
	graphs. The maximizer has
	\(
		M^\star
		\cong
		K_{1,h}\cup(h+1)K_1,
	\)
	whereas the minimizer has
	\(
		M^{\min}
		\cong
		hK_2\cup2K_1.
	\)
	When \(h=1\), the missing-edge graph is the same at both endpoints,
	and the distinction lies entirely in the placement of the two
	exceptional \(X\)--\(Y\) edges: they have a common endpoint in \(Y\)
	at the spectral maximum and distinct endpoints in \(Y\) at the
	spectral minimum.

	Thus, for \(h\ge2\), the matching missing-edge graph proposed by
	Chang--Li--Zhang \cite{ChangLiZhang2026} is the missing-edge graph
	of the spectral minimizer rather than that of the spectral
	maximizer.

	More generally, for every \(h\ge1\), consider the class with
	\(
		M\cong hK_2\cup2K_1.
	\)
	If \(r_y\) denotes the number of exceptional \(X\)--\(Y\) edges
	incident with \(y\), then the variable part of the first
	placement-sensitive coefficient is proportional to
	\(
		\sum_{y\in Y}r_y^2,
		\;
		\sum_{y\in Y}r_y=2h.
	\)
	Consequently, this coefficient is maximized when all \(2h\)
	exceptional edges have a common endpoint in \(Y\), and minimized
	when their \(Y\)-endpoints are pairwise distinct.
\end{remark}

	\subsection{Quotient matrices and the sharp threshold}
	\label{subsec:quotients-final-threshold}

	The resolvent expansion determines the extremal structures. The quotient
	matrices below serve a different purpose: they turn the two extremal
	spectral radii into Perron roots of matrices of bounded order and hence
	make the sharp threshold directly computable.

	The minimum spectral radius can also be expressed as
	the Perron root of an explicit quotient matrix. Assume \(h\ge1\), put
	\(
	a=k-h-1,
	\)
	and, in the construction of \(G^{\min}_{h,k,n}\), let
		\(
		Z=\{y_1,\ldots,y_{2h}\},
		\;
		Y_0=Y\setminus Z,
		\;
		q_{\min}=|Y_0|=n-k-3h-1.
	\)
	With rows and columns ordered as
	\(
	A,\ I,\ W,\ Z,\ Y_0,
	\)
	define
	\begin{equation}
		Q^{\min}_{h,k,n}
		=
		\begin{pmatrix}
			2h-2 & 2 & a & 1 & 0\\
			2h   & 1 & a & 0 & 0\\
			2h   & 2 & a-1 & 2h & q_{\min}\\
			1    & 0 & a & 2h-1 & q_{\min}\\
			0    & 0 & a & 2h & q_{\min}-1
		\end{pmatrix}.
		\label{eq:minimum-quotient}
	\end{equation}
	When \(a=0\), the class \(W\) is empty, and
	\(Q^{\min}_{h,k,n}\) is understood to be the matrix obtained from
	\eqref{eq:minimum-quotient} by deleting the third row and third column.
	
	\begin{proposition}
		\label{prop:minimum-equitable-quotient}
		Fix integers \(h\ge1\) and \(k\ge h+1\), and let
		\(
		n\ge n_{\mathrm{int}}(k,k+h).
		\)
		The partition
		\(
		A,\ I,\ W,\ Z,\ Y_0
		\)
		of \(G^{\min}_{h,k,n}\), after deleting the empty class \(W\) when
		\(a=0\), is equitable. Its class sizes are
		\(
		\bigl(2h,\,2,\,a,\,2h,\,q_{\min}\bigr),
		\)
		and its quotient matrix is \(Q^{\min}_{h,k,n}\). Consequently,
		\[
		\min_{
			G\in\mathcal{L}_{\mathcal{H}}^{2}(n,k,k+h)
		}
		\rho(G)
		=
		\rho\bigl(G^{\min}_{h,k,n}\bigr)
		=
		\rho\bigl(Q^{\min}_{h,k,n}\bigr).
		\]
	\end{proposition}
	
	\begin{proof}
		The graph induced by \(A\) is \(K_{2h}\) with a perfect matching
		removed. Hence every vertex of \(A\) has \(2h-2\) neighbors in \(A\),
		two neighbors in \(I\), \(a\) neighbors in \(W\), and one private
		neighbor in \(Z\).
		
		Each vertex of \(I\) is adjacent to all \(2h\) vertices of \(A\), to
		the other vertex of \(I\), and to every vertex of \(W\), but has no
		neighbor in \(Z\cup Y_0\). Every vertex of \(W\) is universal. Each
		vertex of \(Z\) has exactly one neighbor in \(A\), no neighbor in
		\(I\), and is adjacent to all vertices of
		\(W\cup(Z\setminus\{z\})\cup Y_0\).
		Finally, the vertices of \(Y_0\) have no neighbors in \(X\) and are
		adjacent to every vertex of
		\(W\cup Z\cup(Y_0\setminus\{y\})\).
		
		Thus the displayed partition is equitable, and the corresponding
		neighbor counts give exactly the matrix in
		\eqref{eq:minimum-quotient}. By
		\Cref{lem:equitable-quotient},
		\(\rho(G^{\min}_{h,k,n})=\rho(Q^{\min}_{h,k,n})\).
		The minimization identity follows from
		\Cref{thm:intro-ordinary-minimum}.
	\end{proof}
	
	\subsubsection{The maximum quotient and the global threshold}
	
	For \(h\ge2\), let \(G^\star_{h,k,n}\) be the maximizing graph identified
	in \Cref{thm:intro-ordinary}, and put
	\(
	w=k-h-1,
	\;
	q_{\star}=n-k-2h-1.
	\)
	Write
	\(
	X=\{c\}\mathbin{\dot\cup}L\mathbin{\dot\cup}I,
	\;
	|L|=h,
	\;
	|I|=h+1,
	\)
	where \(c\) is the center and \(L\) is the set of leaves of the star
	missing-edge graph. Under the nested placement, write
	\(
	Y=\{y_1\}\mathbin{\dot\cup}Y_2\mathbin{\dot\cup}Y_3,
	\;
	|Y_2|=h-1,
	\;
	|Y_3|=q_{\star}.
	\)
	Here \(y_1\) is the common \(Y\)-neighbor of all vertices in \(L\), and
	\(\{y_1\}\cup Y_2=N_Y(c)\).
	
	With rows and columns ordered as
	\(
	\{c\},\ L,\ I,\ W,\ \{y_1\},\ Y_2,\ Y_3,
	\)
	define
	\begin{equation}
		Q^\star_{h,k,n}
		=
		\begin{pmatrix}
			0 & 0 & h+1 & w & 1 & h-1 & 0\\
			0 & h-1 & h+1 & w & 1 & 0 & 0\\
			1 & h & h & w & 0 & 0 & 0\\
			1 & h & h+1 & w-1 & 1 & h-1 & q_{\star}\\
			1 & h & 0 & w & 0 & h-1 & q_{\star}\\
			1 & 0 & 0 & w & 1 & h-2 & q_{\star}\\
			0 & 0 & 0 & w & 1 & h-1 & q_{\star}-1
		\end{pmatrix}.
		\label{eq:maximum-quotient}
	\end{equation}
	When \(w=0\), the class \(W\) is empty, and
	\(Q^\star_{h,k,n}\) is understood to be the matrix obtained from
	\eqref{eq:maximum-quotient} by deleting the fourth row and fourth
	column.
	
	\begin{proposition}
		\label{prop:maximum-equitable-quotient}
		Fix integers \(h\ge2\) and \(k\ge h+1\), and let
		\(n\ge n_{\mathrm{int}}(k,k+h)\). The partition
		\(
		\{c\},\ L,\ I,\ W,\ \{y_1\},\ Y_2,\ Y_3
		\)
		of \(G^\star_{h,k,n}\), after deleting the empty class \(W\) when
		\(w=0\), is equitable. Its class sizes are
		\(
		\bigl(1,\,h,\,h+1,\,w,\,1,\,h-1,\,q_{\star}\bigr),
		\)
		and its quotient matrix is \(Q^\star_{h,k,n}\). Consequently,
		\(
		\rho\bigl(G^\star_{h,k,n}\bigr)
		=
		\rho\bigl(Q^\star_{h,k,n}\bigr).
		\)
	\end{proposition}
	
	\begin{proof}
		We verify the neighbor counts in the stated order of the partition
		classes.
		
		The center \(c\) has no neighbor in \(L\), is adjacent to every
		vertex of \(I\cup W\), and has \(Y\)-neighborhood
		\(\{y_1\}\cup Y_2\). Each vertex of \(L\) is adjacent to the other
		\(h-1\) vertices of \(L\), every vertex of \(I\cup W\), and \(y_1\),
		but is not adjacent to \(c\) or to any vertex of \(Y_2\cup Y_3\).
		
		Each vertex of \(I\) is adjacent to \(c\), all \(h\) vertices of
		\(L\), the other \(h\) vertices of \(I\), and every vertex of \(W\),
		but has no neighbor in \(Y\). Every vertex of \(W\) is universal.
		
		The vertex \(y_1\) is adjacent to \(c\), all vertices of \(L\), and
		every vertex of
		\(
		W\cup Y_2\cup Y_3.
		\)
		Each vertex of \(Y_2\) is adjacent to \(c\), every vertex of \(W\),
		the vertex \(y_1\), the other \(h-2\) vertices of \(Y_2\), and every
		vertex of \(Y_3\). Finally, every vertex of \(Y_3\) is adjacent to
		all vertices of
		\(
		W\cup\{y_1\}\cup Y_2\cup
		\bigl(Y_3\setminus\{y\}\bigr),
		\)
		and has no neighbor in \(X\).
		
		These counts give precisely the matrix in
		\eqref{eq:maximum-quotient}. The partition is equitable, and hence
		\Cref{lem:equitable-quotient} gives
		\(
		\rho\bigl(G^\star_{h,k,n}\bigr)
		=
		\rho\bigl(Q^\star_{h,k,n}\bigr).
		\)
	\end{proof}
	
    For the two boundary cases, let \(G^\star_{1,k,n}\) and
    \(G^\star_{0,k,n}\) denote the maximizing graphs identified in
    \Cref{lem:boundary-h}.
    
    Suppose first that \(h=1\). Let \(uv\) be the unique edge of the
    missing-edge graph, let \(I\) be the set of the other two vertices of
    \(X\), and let \(y_1\in Y\) be the common \(Y\)-neighbor of \(u\) and
    \(v\). Put
	    \(
	    Y'=Y\setminus\{y_1\},
	    \;
	    |Y'|=n-k-3.
    \)
    With rows and columns ordered as
    \(
    \{u,v\},\ I,\ W,\ \{y_1\},\ Y',
    \)
    define
    \begin{equation}
    	Q^\star_{1,k,n}
    	=
    	\begin{pmatrix}
    		0 & 2 & k-2 & 1 & 0\\
    		2 & 1 & k-2 & 0 & 0\\
    		2 & 2 & k-3 & 1 & n-k-3\\
    		2 & 0 & k-2 & 0 & n-k-3\\
    		0 & 0 & k-2 & 1 & n-k-4
    	\end{pmatrix}.
    	\label{eq:maximum-quotient-h1}
    \end{equation}
    When \(k=2\), the class \(W\) is empty, and
    \(Q^\star_{1,2,n}\) is understood to be the matrix obtained from
    \eqref{eq:maximum-quotient-h1} by deleting the third row and third
    column before substituting \(k=2\).
    
    When \(h=0\), the partition
    \(
    X,\ W,\ Y
    \)
    has class sizes
    \(
    2,\ k-1,\ n-k-1,
    \)
    and its quotient matrix is
    \begin{equation}
    	Q^\star_{0,k,n}
    	=
    	\begin{pmatrix}
    		1 & k-1 & 0\\
    		2 & k-2 & n-k-1\\
    		0 & k-1 & n-k-2
    	\end{pmatrix}.
    	\label{eq:maximum-quotient-h0}
    \end{equation}
	    Both displayed partitions are equitable. Hence
	    \Cref{lem:equitable-quotient} gives
	    \(\rho(G^\star_{1,k,n})=\rho(Q^\star_{1,k,n})\) and
	    \(\rho(G^\star_{0,k,n})=\rho(Q^\star_{0,k,n})\).
    
    Thus, in the following corollary, \(G^\star_{h,k,n}\) denotes the
    maximizing graph defined in the appropriate one of the three cases
    \(h=0\), \(h=1\), and \(h\ge2\).
    
    \begin{corollary}
    	\label{cor:final-threshold}
    	Fix integers \(k\ge2\) and \(k\le\delta\le2k-1\), put
    	\(
    	h=\delta-k,
    	\)
    	and let
    	\(
    	n\ge n_0(k,\delta),
    	\)
    	where \(n_0(k,\delta)\) is defined in
    	\eqref{eq:n0-definition}. Then every
    	\(G\in\mathcal{G}_{n,\delta}\) satisfies
    	\(
    	\tau(G)\le k-1
    	\;\Longrightarrow\;
    	\rho(G)\le\rho\bigl(G^\star_{h,k,n}\bigr),
    	\)
    	with equality if and only if
    	\(
    	G\cong G^\star_{h,k,n}.
    	\)
    	Equivalently,
    	\(
    	\rho(G)\ge\rho\bigl(G^\star_{h,k,n}\bigr)
    	\;\Longrightarrow\;
    	\tau(G)\ge k
    	\ \text{or}\
    	G\cong G^\star_{h,k,n}.
    	\)
    	In particular,
    	\(
    	\rho(G)>\rho\bigl(G^\star_{h,k,n}\bigr)
    	\;\Longrightarrow\;
    	\tau(G)\ge k.
    	\)
    	
    	The sharp threshold \(\rho(G^\star_{h,k,n})\) is the Perron root of
    	\(Q^\star_{0,k,n}\), \(Q^\star_{1,k,n}\), or
    	\(Q^\star_{h,k,n}\), according as \(h=0\), \(h=1\), or \(h\ge2\),
    	where the corresponding matrices are given in
    	\eqref{eq:maximum-quotient-h0},
    	\eqref{eq:maximum-quotient-h1}, and
    	\eqref{eq:maximum-quotient}, respectively.
    \end{corollary}
    
    \begin{proof}
    	Let
    	\(
    	G\in\mathcal{G}_{n,\delta}\;
    	\tau(G)\le k-1.
    	\)
    	Suppose first that
    	\(
    	\rho(G)\ge\rho\bigl(G^\star_{h,k,n}\bigr).
    	\)
    	Since \(n\ge n_0(k,\delta)\ge n_{\mathrm{CLZ}}(k,\delta)\),
    	Chang--Li--Zhang's Theorem~1.5 implies that
    	\(
    	G\in\mathcal{L}^{*}(n,k,\delta).
    	\)
    	By \Cref{thm:intro-ordinary},
    	\(G^\star_{h,k,n}\) is the unique member of
    	\(\mathcal{L}^{*}(n,k,\delta)\) up to isomorphism. Therefore
    	\(
    	G\cong G^\star_{h,k,n}
    	\)
    	and
    	\(
    	\rho(G)=\rho\bigl(G^\star_{h,k,n}\bigr).
    	\)
    	It follows that every graph satisfying \(\tau(G)\le k-1\) obeys
    	\(
    	\rho(G)\le\rho\bigl(G^\star_{h,k,n}\bigr),
    	\)
    	with equality if and only if
    	\(G\cong G^\star_{h,k,n}\). Equivalently, if
    	\(\rho(G)\ge\rho(G^\star_{h,k,n})\), then either
    	\(\tau(G)\ge k\) or \(G\cong G^\star_{h,k,n}\).
    	The strict-threshold implication is an immediate special case.
    	
    	Finally,
    	\Cref{prop:packing-obstruction} gives
    	\(
    	\tau\bigl(G^\star_{h,k,n}\bigr)\le k-1,
    	\)
    	so the threshold is attained and is therefore sharp. Its quotient
    	matrix description follows from
    	\eqref{eq:maximum-quotient-h0},
    	\eqref{eq:maximum-quotient-h1}, and
    	\Cref{prop:maximum-equitable-quotient}.
    \end{proof}

\section{Conclusion}

For every fixed pair \((k,\delta)\) satisfying
\(k\le\delta\le2k-1\), we have determined both spectral endpoints of
\(\mathcal{L}_{\mathcal{H}}^{2}(n,k,\delta)\) for all sufficiently large
\(n\). The maximum is governed by concentration: its missing-edge graph is a
star for \(h=\delta-k\ge2\), and its exceptional edges have a nested
placement. The minimum is governed by dispersion: its missing-edge graph is a
matching, and its exceptional edges have pairwise distinct endpoints in the
large clique. The boundary cases \(h=0\) and \(h=1\) fit the same principle.

The proof also supplies exact bounded-order quotient matrices for the two
extremal graphs. Combining the maximizing structure with the reduction of
Chang--Li--Zhang gives the sharp adjacency-spectral threshold for forcing
\(k\) edge-disjoint spanning trees in the stated minimum-degree range, subject
to the same sufficiently-large-order hypothesis.
\section*{Acknowledgments}

The authors would like to thank the members of the graph theory group at
Central China Normal University for helpful discussions and comments on an
earlier version of this manuscript.
\section*{Data Availability Statement}

No datasets were generated or analyzed in this study.


\small
\begin{thebibliography}{99}
		
		\bibitem{CioabaWong2012}
		S. M. Cioab\u{a} and W. Wong,
		Edge-disjoint spanning trees and eigenvalues of regular graphs,
		\emph{Linear Algebra Appl.} 437 (2012), 630--647.
		
		\bibitem{ChangLiZhang2026}
		L. Chang, S. Li and M. Zhang,
		Edge-disjoint spanning trees, eigenvalues, and size of graphs,
		\emph{J. Graph Theory} (2026), 1--26,
		doi:10.1002/jgt.70087.
		
		\bibitem{CvetkovicRowlinsonSimic2010}
		D. Cvetkovi\'{c}, P. Rowlinson and S. Simi\'{c},
		\emph{An Introduction to the Theory of Graph Spectra},
		Cambridge University Press, Cambridge, 2010.
		
		\bibitem{FanGuLin2023}
		D. Fan, X. Gu and H. Lin,
		Spectral radius and edge-disjoint spanning trees,
		\emph{J. Graph Theory} 104 (2023), 697--711,
		doi:10.1002/jgt.22996.
		
		\bibitem{GuLaiLiYao2016}
		X. Gu, H.-J. Lai, P. Li and S. Yao,
		Edge-disjoint spanning trees, edge connectivity and eigenvalues in graphs,
		\emph{J. Graph Theory} 81 (2016), 16--29,
		doi:10.1002/jgt.21857.
		
		\bibitem{LiShi2013}
		G. Li and L. Shi,
		Edge-disjoint spanning trees and eigenvalues of graphs,
		\emph{Linear Algebra Appl.} 439 (2013), 2784--2789,
		doi:10.1016/j.laa.2013.08.041.
		
		\bibitem{LiuHongLai2014}
		Q. Liu, Y. Hong and H.-J. Lai,
		Edge-disjoint spanning trees and eigenvalues,
		\emph{Linear Algebra Appl.} 444 (2014), 146--151,
		doi:10.1016/j.laa.2013.11.039.
		
		\bibitem{LiuHongGuLai2014}
		Q. Liu, Y. Hong, X. Gu and H.-J. Lai,
		Note on edge-disjoint spanning trees and eigenvalues,
		\emph{Linear Algebra Appl.} 458 (2014), 128--133,
		doi:10.1016/j.laa.2014.05.044.
		
		\bibitem{NashWilliams1961}
		C. St. J. A. Nash-Williams,
		Edge-disjoint spanning trees of finite graphs,
		\emph{J. London Math. Soc.} 36 (1961), 445--450.
		
		\bibitem{Palmer2001}
		E. M. Palmer,
		On the spanning tree packing number of a graph: a survey,
		\emph{Discrete Math.} 230 (2001), 13--21.
		
		\bibitem{Tutte1961}
		W. T. Tutte,
		On the problem of decomposing a graph into \(n\) connected factors,
		\emph{J. London Math. Soc.} 36 (1961), 221--230.
		
	\end{thebibliography}
\end{document}